\documentclass[preprint,11pt,3p]{elsarticle}
\usepackage{verbatim}
\usepackage{amsmath,amsthm,amsfonts,amssymb}
\usepackage{algorithm}
\newcommand{\mat}{\mathbf{mat}}
\newcommand{\rank}{{rank}_{M}}
\newcommand{\ind}{{ind}_{M}}
\usepackage{enumerate}
\usepackage{graphics}
\usepackage{algpseudocode}
\usepackage{cases}
\usepackage[hidelinks]{hyperref}
\usepackage[normalem]{ulem}
\usepackage{xcolor,sectsty}
\usepackage{mathtools}
\usepackage{rotating}
\newtheorem{theorem}{Theorem}[section]
\newtheorem{lemma}[theorem]{Lemma}
\newtheorem{definition}[theorem]{Definition}
\newtheorem{proposition}[theorem]{Proposition}
\newtheorem{corollary}[theorem]{Corollary}
\newtheorem{example}[theorem]{Example}

\newcommand{\mb}{\mathbb}
\newcommand{\mc}[1]{\mathcal {#1}}
\newcommand{\m}{*_{M}}
\usepackage{graphicx}
\usepackage{mathrsfs}
\usepackage{multirow}
\newcommand{\rg}{{\mathscr{R}}_M}
\usepackage{hyperref}
\usepackage{subfigure}

\newcommand{\Break}{\State \textbf{break} }

\newcommand{\nl}{{\mathscr{N}}_M}
\usepackage{multicol}

\usepackage{tikz,xcolor}
\definecolor{lime}{HTML}{A6CE39}
\definecolor{lightblue}{rgb}{0.0, 0.0, 0.5}
\DeclareRobustCommand{\orcidicon}{%
	\begin{tikzpicture}
	\draw[lime, fill=lime] (0,0)
	circle [radius=0.16]
	node[white] {{\fontfamily{qag}\selectfont \tiny ID}};
	\draw[white, fill=white] (-0.0625,0.095)
	circle [radius=0.007];
	\end{tikzpicture}
	\hspace{-2mm}
}
\foreach \x in {A, ..., Z}{
	\expandafter\xdef\csname orcid\x\endcsname{\noexpand\href{https://orcid.org/\csname orcidauthor\x\endcsname}{\noexpand\orcidicon}}
}

\begin{document}

\begin{frontmatter}

\title{
$M$-QR decomposition and hyperpower iterative methods for computing outer inverses of tensors
}
\author{Ratikanta Behera$^{a,2}$, Krushnachandra Panigrahy$^{a,3}$, Jajati Keshari Sahoo$^b$, Yimin Wei$^c$\footnote{Corresponding author: Yimin Wei.} }
\vspace{.3cm}

\address{  
$^{a}$Department of Computational and Data Sciences \\
Indian Institute of Science, Bangalore, 560012, India.\\
                        \textit{E-mail$^{2}$}: \texttt{ratikanta@iisc.ac.in}\\
                        \textit{E-mail$^{3}$}: \texttt{kcp.224@gmail.com}\\
                        \vspace{.3cm}
      $^{b}$Department of Mathematics \\ Birla Institute of Technology and Science, Pilani, K K Birla Goa Campus, Zuarinagar, Sancoale, Goa 403726, India\\
        \textit{E-mail}: \texttt{jksahoo@goa.bits-pilani.ac.in}\\  
        \vspace{.3cm}
$^{c}$ School of Mathematics Sciences and Shanghai Key Laboratory of Contemporary Applied Mathematics,\\
Fudan University, Shanghai 200433, P.R. China.\\
\textit{E-mail}: \texttt{ymwei@fudan.edu.cn}
}
    
\begin{abstract}
The outer inverse of tensors plays increasingly significant roles in computational mathematics, numerical analysis, and other generalized inverses of tensors. In this paper, we compute outer inverses with prescribed ranges and kernels of a given tensor through tensor QR decomposition and hyperpower iterative method under the $M$-product structure, which is a family of tensor–tensor products, generalization of the $t$-product and $c$-product, allows us to suit the physical interpretations across those different modes. We discuss a theoretical analysis of the nineteen-order convergence of the proposed tensor-based iterative method. Further, we design effective tensor-based algorithms for computing outer inverses using $M$-QR decomposition and hyperpower iterative method. The theoretical results are validated with numerical examples demonstrating the appropriateness of the proposed methods. In addition, we examine the application of $M$-QR decomposition and hyperpower iterative methods to improve image compression efficiency and improve deblurring results in digital image processing.
\end{abstract}

\end{frontmatter}

\section{Introduction}
Generalized inverses of tensors \cite{che20, wei16} play a vital role in the field of science and engineering such as image processing, data sciences, statistics, and machine learning. These are extensions of regular inverses to a class of tensors larger than the class of nonsingular tensors, which reduces to the usual inverse when the tensor is nonsingular. Tensor products, the t-product, introduced by Kilmer and Martin \cite{kilmer11}, have significantly impacted many areas of science and engineering, such as computer vision \cite{hao2013, hu16}, image processing \cite{Kilmer13, Martin13}, signal processing \cite{chan16, long19}, data completion and denoising \cite{hu16, wang19}, low-rank tensor recovery \cite{kong18,wang19}, and robust tensor PCA \cite{kong18, liu18}. In 2015, Kernfeld, Kilmer, and Aeron introduced a new tensor product called the $M$-product, which generalizes the previously established $t$-product by Kilmer and Martin \cite{kilmer11}. They demonstrated that the well-known $c$-product of tensors is a particular case of the $M$-product. Recently, Jin {\it et al.} \cite{jin2023} introduced the Moore-Penrose inverse of tensors using $M$-product. The authors discussed the minimum-norm solutions for tensor equations. They put forth some necessary and sufficient conditions for the reverse-order law for the Moore-Penrose inverse of tensors, all within the framework of the $M$-product. Numerical and symbolic computations of generalized inverses can be referred to  \cite{wei18}. The singular value decomposition of dual matrices and the dual Moore-Penrose inverse with applications to traveling wave identification in the brain has been developed in \cite{wei24}.

Jin {\it et al.} \cite{jin2017} introduced the Moore–Penrose inverse of the tensors in the framework of $t$-product and established its existence and uniqueness using the technique of fast Fourier transform. In addition, an algorithm was also constructed to compute the Moore–Penrose inverse of an arbitrary tensor. In 2021, Miao {\it et al.} \cite{Mia21} investigated the $t$-Drazin inverse and $t$-group inverse. Also, they proposed the $t$-Jordan canonical form and used it to express the $t$-Drazin and $t$-group inverses. Studies on the $t$-eigenvalue and eigenvectors have been discussed in \cite{chen23, Weihuixia21}. Cui and Ma \cite{cui21} established the perturbation bounds of the  $t$-Drazin and $t$-group inverses. Cong and Ma \cite{cong2022} defined the acute perturbation for the Moore-Penrose inverse of tensors under the $t$-product by $t$-spectral radius. Also, they presented the equivalent relation between acute perturbation and stable perturbation for the Moore-Penrose inverse of tensors.  In the same year, Liu and Ma \cite{liu2022} discussed the existence of the dual-core inverse based on the $t$-product and characterized the dual tensor with dual index one. Also, they established the concepts of dual Moore–Penrose inverse and group inverse. Behera {\it et al.} \cite{behera2022} studied several generalized inverses of tensors over commutative and noncommutative rings using $t$-product. Further,  Behera {\it et al.} \cite{behera2023} introduced the outer inverse of tensors based on $t$-product and supplied a $t$-QR decomposition algorithm to compute outer generalized inverses. The study of a few tensor decompositions such as $t$-UTV, $t$-Schur, $t$-CUR, and $t$-LU can be found in \cite{che22,chen23, chen22, miao20, zhu22}.

One of the most powerful iterative methods to compute the outer inverse of a matrix is hyperpower iteration (HPI$p$) with arbitrary order of convergence $p \geq 2$. A standard $p^th$-order HPI$p$ method requires $p$ number of matrix–matrix products (MMP). Since matrix products are high computational cost operations, it is natural to investigate an appropriate and simple factorization of the standard HPI$p$ method to allow achieving the $p^{th}$-order of convergence with fewer MMP per iteration. The potential of HPI$p$ methods depends on the {\it Informational Efficiency Index} (IEI) \cite{traub1964} and the {\it Computational Efficiency Index} (CEI) \cite{traub1964}. If  $p$ is the order of convergence and $n$ is the number of MMP, then
\begin{equation*}
    IEI=\dfrac{p}{n} {\rm~ and~} CEI=p^{1/n}.
\end{equation*}
A higher CEI value indicates that the method is more computationally efficient. In 2014, Stanimirovi\'c and Soleymani \cite{stanimirovic2014} studied the $7^{th}$ order HPI$7$ method and factorize it smartly so that it requires only $5$ MMP per iteration (${\rm CEI}=7^{1/5}\approx1.476$). In 2015, M. D. Petkovi\'c and M. S. Petkovi\'c \cite{hyper-1} investigated HPI$5$ and HPI$9$ methods with CIE respectively, $5^{1/4}\approx1.495$  and $9^{1/5}\approx1.552$, for computing outer inverses, which requires $4$ and $5$ number of MMP per iteration respectively. In the same year, Soleymani {\it et al.} \cite{hyper-2} studied a family of HPI methods with orders ranging from $10$ to $19$ order and verified that the $19^{th}$-order iterative method is factorized involving only $8$ number of MMP per iteration (${\rm CEI}=19^{1/8}\approx 1.445$). Among $10^{th}$ to $19^{th}$-order of convergence, the authors verified the higher efficiency index of the $19^{th}$-order method. An efficient factorization of the hyperpower method with the order of convergence $30$ (${\rm CEI}=30^{1/9}\approx1.459$) was recently constructed and analyzed in \cite{sharifi2015}. Recently, Ma {\it et al.} \cite{ma2022} improved the hyperpower method of order nineteen (HPI19) by using only seven MMP per iteration (${\rm CEI}=19^{1/7}\approx 1.523$).

The proposed approach establishes a generalized transform domain framework that extends beyond the Fourier transform domain by developing tensor QR decomposition with a comparative analysis of computations involving the $t$-product and the $c$-product. In addition, a novel tensor-based hyperpower iterative (HPI) method is developed to calculate outer inverses, with a systematic performance evaluation conducted against the established t-QR decomposition. The investigation also includes specialized cases that involve the computation of the Drazin inverses and the Moore-Penrose inverses. The main contributions of this paper are as follows.

\begin{itemize}
\item Introduce tensor $M$-QR decomposition with prescribed ranges and kernels for computing outer inverses of tensors based on $M$-product.  

\item Introduce hyperpower iterative method for computing outer inverses of tensors via $M$-product.  

\item Design effective tensor-based algorithms for computing outer inverses using $M$-QR decomposition and hyperpower iterative method.

\item Present an extensive theoretical analysis of the $19^{\text{th}}$-order convergence of the proposed tensor-based iterative method. 

\item The theoretical results are validated with numerical examples that demonstrate the appropriateness of the proposed methods.

\item We examine the application of $M$-QR decomposition and hyperpower iterative methods for image compression and deblurring.
\end{itemize}
The outline of the paper is as follows. In Section 2, we introduce the construction of the $M$-product. Then, we present a few results on $\mat$ and $\mat^{-1}$ along with tensor $M$-QR decomposition in Section 3. In this section, we design a few algorithms to compute $M$-QR decomposition and the outer inverse of tensors; then, we discuss a few examples to validate our proposed decomposition. Section 4 discusses the tensor-based hyperpower iteration method, designing an efficient algorithm based on $M$-product, and computing different generalized inverses as a particular case of the outer inverse. 
Section 5 validates the theoretical results through a few numerical examples, demonstrating the appropriateness and effectiveness of the proposed algorithms. Section 6 examines practical applications of $M$-QR decomposition and hyperpower iterative methods for image compression and deblurring tasks. Finally, in Section 7, we present the conclusions of this work and discuss potential directions for future research.

\section{Preliminaries}\label{sec:pre}
Let $\mc{A}^{(i)}=\mc{A}(:,:,i)$ represents the $i^{th}$ frontal slice of $\mc{A} \in \mb{F}^{m \times n \times p}$, where $\mb{F}$ is either $\mathbb{R}$ or $\mb{C}$. The part of a tensor obtained by fixing any two indices of a third-order tensor is a tube. Thus $\mc{A}(i, j,:)$ or $ \mc{A}(i,:,k)$ or $ \mc{A}(:,j,k)$ represents tube fiber of $\mc{A} \in \mb{F}^{m \times n \times p}$ for some fixed $i,j,k$.
The mode-$3$ fibers, which are $1 \times 1 \times p$ tensors are denoted with lowercase bold letters, e.g.,  $\mathbf{a}\in \mb{F}^{1\times 1\times n}$.
\begin{definition}{\rm(\cite{kolda09})}
   Let $\mc{A}\in {\mb{F}}^{m\times n \times k}$ be a tensor and $B\in{\mb{F}^{p\times k}}$ be a matrix. The $3$-mode product of $\mc{A}$ with $B$ is  denoted by $\mc{A}\times_{3}B\in{\mb{F}}^{m\times n\times p}$ and element-wise defined as 
      \begin{equation*}
          (\mc{A}\times_{3} B)_{ijl}=\sum_{s=1}^{k}a_{ijs}b_{ls}\quad i=1,2,\ldots, m, ~j=1,2,\ldots, n,~l=1,2\ldots, p.
      \end{equation*}  
\end{definition}
In particular, if $\mc{A}\in \mb{F}^{m\times n\times p}$ and $M\in \mb{F}^{p\times p}$ are invertible matrices, we will use `hat' notation to denote a tensor in the transform domain specified by $M$, that is,
\begin{equation*}
    \hat{\mc{A}}:=\mc{A}\times_{3}M \in \mb{R}^{m\times n \times p}.
\end{equation*}
The `hat' notation should be understood in context relative to the $M$ applied. From here onwards, we will assume $M$ is an invertible matrix. Next, we define the `$\mat$' operation for transforming a tensor into a matrix with respect to $M$.
\begin{definition}[Definition 2.6 and Lemma 3.1, \cite{kernfeld2015}]
Let $\mc{A}\in \mb{F}^{m\times n\times p}$,  $M\in \mb{F}^{p\times p}$ and $(\hat{\mc{A}})^{(i)}\in  \mb{F}^{m\times n}$ be the $i^{th}$ frontal slice of $\hat{\mc{A}}$, for $i=1,2,\ldots, p$. Then, $\mat:\mb{F}^{m\times n\times p} \rightarrow \mb{F}^{mp\times np}_M$ is defined as 

\begin{equation*}
\mat(\mc{A})=\begin{pmatrix}
(\hat{\mc{A}})^{(1)} &0 &\cdots &0\\
0 & (\hat{\mc{A}})^{(2)}  & \cdots &0\\
\vdots  &   \vdots  &    \ddots & \vdots\\
0& 0&  \cdots & (\hat{\mc{A}})^{(p)}
\end{pmatrix}.
\end{equation*}
\end{definition}
\noindent Let us consider $\mb{F}^{mp\times np}_M = \left\{ A\in \mb{F}^{mp\times np}: A \mbox{ is of the form } \begin{pmatrix}
     {A}_1 &  & O & \\
     & {A}_2 &  & \\
     & &  \ddots &\\ 
    & O&  & {A}_{p}
    \end{pmatrix} \right\}$.   The inverse operation of $\mat$, i.e., $\mat^{-1}: \mb{F}^{mp\times np}_M \rightarrow  \mb{F}^{m \times n\times p}$ is defined by 
    \begin{equation*}
        \mat^{-1}(A)=\mc{A},
    \end{equation*}
     where $\mc{A}=\hat{\mc{A}}\times_3M^{-1}$ and $\hat{\mc{A}}\in\mb{F}^{m \times n\times p} \mbox{ with } (\hat{\mc{A}})^{(i)}=A_i,~i=1,2,\ldots, p.$
Thus, we can represent a tensor $\mc{A}\in\mb{F}^{m\times n\times p}$ as follows:
\[\mc{A}=\mat^{-1}(\mat(\mc{A})).\]
Now, the $M$-product of two tensors is defined with the help of `$\mat$' and `$\mat^{-1}$'.
 \begin{definition}{\rm(\cite{kernfeld2015})}
    Let $\mc{A}\in \mb{F}^{m\times n\times p}$,  $\mc{B}\in \mb{F}^{n\times k\times p}$ and $M\in \mb{F}^{p\times p}$. Then, the M-product of tensors $\mc{A}$ and $\mc{B}$ is denoted by $\mc{A}*_{M}\mc{B}\in \mb{F}^{m\times k\times p}$ and defined as
\[\mc{A}*_{M}\mc{B}=\mat^{-1}(\mat(\mc{A})\mat(\mc{B})).
\]
\end{definition}
Note that if we choose $M$ as the unnormalized DFT matrix (MATLAB command $dftmtx(n)$ can be used to generate the DFT matrix of order $n$), then the $M$-product coincides with the $t-$product \cite{kilmer11}. The cosine transform product ($c-$product \cite{kernfeld2015}) is similarly obtained by taking the matrix $M=W^{-1}C(I + Z)$, where $C$ is the DCT matrix  (MATLAB command $dctmtx(n)$) of order $n$, $Z=diag(ones(n-1,1),1)$, and $W=diag(C(:, 1))$. We denote the matrix $W^{-1}C(I + Z)$ by $M_1$. Thus, $c-$product is obtained while choosing $M=M_1$.

\begin{definition}{\rm(\cite{kernfeld2015})}
  Let $\mc{A}\in \mb{F}^{m\times m\times p}$  and $M\in\mb{F}^{p\times p}$. Then, $\mc{A}$ is said to be an identity tensor if $(\mc{A}\times_{3}M)^{({i})}=I_m$ for $i=1,2,\ldots, p$ and $I_m$ is the identity matrix of order $m$. 
\end{definition}
Equivalently, we can say that $\mc{A}$ is an identity tensor if the block diagonals of $\mat(\mc{A})$ are identity matrices of order $m$. We denote the identity tensor $\mc{I}\in \mb{F}^{m\times m\times p}$ by $\mc{I}_{mmp}$. Further, if the block diagonals of $\mat(\mc{A})$ are permutation  matrices of order $m$ then we call $\mc{A}$ a permutation tensor.

\begin{definition}{\rm(\cite{kernfeld2015})}
Let $\mc{A}\in \mb{F}^{m\times n\times p}$  and $M\in\mb{F}^{p\times p}$.  Then, the transpose conjugate of $\mc{A}$ is denoted by $\mc{A}^*$, which can be obtained from the relation  $(\mc{A}^*\times_{3}M)^{(i)}= ((\mc{A}\times_{3}M)^{({i})})^*$ for $i=1,2,\ldots, p$. Equivalently, $\mc{A}^*=\mat^{-1}(\mat(A)^*)$.
\end{definition}
\begin{definition}{\rm(\cite{kernfeld2015})}
  Let $\mc{A}\in \mb{F}^{m\times m\times p}$  and $M\in\mb{F}^{p\times p}$. If there exists a tensor $\mc{X}$ such that $\mc{A}*_M\mc{X}=\mc{I}_{mmp}=\mc{X}*_M\mc{A}$ then $\mc{A}$ is invertible and $\mc{A}^{-1}=\mc{X}$.
\end{definition}

\begin{definition}
Let $\mc{A} \in \mb{F}^{m\times n\times p}$ and $M\in\mb{F}^{p\times p}$. Then
 \begin{enumerate}
\item[(a)] the range space of $\mc{A}$ is denoted by $\rg(\mc{A})=\{\mat^{-1}({Y}):    Y\in R(\mat(\mc{A}))\}$, where $R(\mat(\mc{A}))={diag}\left(R(A_1), R(A_2),\ldots, R(A_p)\right)$. Here $R(A)$ represents the range space of the matrix $A$ and $A_i=\left(\hat{\mc{A}}\right)^{(i)}$ for $i=1,2,\ldots,p$.
 \item[(b)] the null space of $\mc{T}$ is denoted by $\nl(\mc{A})=\{\mat^{-1}(Z): Z\in N(\mat(\mc{A}))\}$, where $N(\mat(\mc{A}))= diag(N(A_1), N(A_2),\ldots, N(A_p))$.  Here $N(A)$ represents the null space of the matrix $A$ and $A_i=\left(\hat{\mc{A}}\right)^{(i)}$ for $i=1,2,\ldots,p$.
\item[(c)] the Frobenius norm of $\mc{A}$ is defined by $\displaystyle ||\mc{A}||_{F}:=\sqrt{\sum_{i=1}^m\sum_{j=1}^n\sum_{k=1}^p|a_{ijk}|^2}$.  
\item[(d)] the rank of $\mc{A}$ is denoted by $\rank(\mc{A}):=rank(\mat(\mc{A}))$. 
 \end{enumerate}
\end{definition}

\begin{definition}
Let $M\in\mb{F}^{p\times p}$ and $\mc{A} \in \mb{F}^{m\times n\times p}$ with $\rank(\mc{A})=rp$. Then $\mc{A}=\mc{B}*_M\mc{C}$ is called full-rank decomposition of $\mc{A}$, where $\mc{B}\in \mb{F}^{m\times r\times p}$ and $\mc{C}\in \mb{F}^{r\times n\times p}$.
\end{definition}

\begin{definition}
Let  $\mc{A} \in \mb{F}^{m\times m\times p}$ and $M\in\mb{F}^{p\times p}$. Then
\begin{enumerate}
   \item[(a)] the index of $\mc{A}$ is the smallest non negative integer $k$ satisfying $\rank(\mc{A}^k)=\rank(\mc{A}^{k+1})$. The index of $\mc{A}$ is denoted by $\ind(\mc{A})$.
\item[(b)]  $\mc{A}$ is symmetric (Hermitian) if $\mat(\mc{A})=\mat(\mc{A})^T$ ($\mat(\mc{A})=\mat(\mc{A})^*)$.
 \end{enumerate}
\end{definition}

\begin{center}
Tensor equations with respect to generalized inverses.
\vspace{.2cm}
\begin{tabular}{|c|l|c|l|c|l|c|l|}
 \hline
Label & Equation  & Label & Equation & Label & Equation \\\hline 
$(1)$ & $\mc{A}*_M\mc{X}*_M\mc{A}=\mc{A}$  & $(2)$ & $\mc{X}*_M\mc{A}*_M\mc{X}=\mc{X}$ & $(4)$ & $(\mc{X}*_M\mc{A})^*=\mc{X}*_M\mc{A}$ \\\hline
$\left(1^k\right)$ &   $\mc{A}^{k+1}*_M\mc{X}=\mc{A}^k$ & $(3)$ & $(\mc{A}*_M\mc{X})^*=\mc{A}*_M\mc{X}$ &$(5)$ &  $\mc{A}*_M\mc{X} = \mc{X}*_M\mc{A}$ \\\hline
\end{tabular}
\end{center}

Let $\mc{A}\{\zeta\}$ be the set of all $\{\zeta\}$-inverses of $\mc{A}$, where $\zeta\in \{1, 1^k, 2, 3, 4,5 \}$. For instance, a tensor $\mc{X} \in \mb{F}^{n\times m\times p} $ is called a  $\{1\}$-inverse of $\mc{A} \in \mb{F}^{m\times n\times p}$, if $\mc{X}$ satisfies $(1)$, i.e., $\mc{A}*_M\mc{X}*_M\mc{A}=\mc{A}$ and we denote $\mc{X}$ by $\mc{A}^{-}$.  Similarly, an outer inverse (or $\{2\}$-inverse) of $\mc{A}$ is denoted by $\mc{A}^{(2)}$.  Further, a tensor $\mc{Y}$ from the set $\mc{A}\{2\}$ satisfying $\rg(\mc{X})=\rg(\mc{B})$ and  $\nl(\mc{X}) = \nl(\mc{C})$ is denoted by $\mc{A}^{(2)}_{\rg(\mc{B}),\nl(\mc{C})}$. 
In case of matrices, we recall the following characterization of the outer inverse, which will be useful in proving for higher-order tensors.  

  \begin{lemma}{\rm (\cite[Corollary 2.1]{stanimirovic2014})}\label{lem2.15-rn}
       Let $A \in \mb{F}^{m\times n}$ and $W \in  \mb{F}^{n\times m}$ such that $rank(WA)=rank(W)\leq rank(A)$. Assume that $X=A^{(2)}_{R(W),N(W)}$. Then 
       \begin{center}
        $ZAX=Z$ and  $XAZ=Z$ if and only $R(Z)\subseteq R(W)$ and $N(W)\subseteq N(Z)$.   
       \end{center}
       \end{lemma}
If  $\mc{X} \in \mc{A}\{1,2,3,4\}$, then  $\mc{X}$ is called the Moore-Penrose inverse \cite{jin2023} of $\mc{A}$ and denoted by $\mc{A}^\dagger$. If $\ind(\mc{A}) = k$ and $\mc{X} \in \mc{A}\{1^k, 2,5\}$, then  $\mc{X}$ is called  the Drazin inverse of $\mc{A}$ and is denoted by $\mc{A}^D$. In the case of $k=1$, we call $\mc{X}$ the group inverse of $\mc{A}$ and it is denoted by $\mc{A}^{\#}$. 
\section{Tensor $M$-QR decomposition }
In this section, we first discuss a few properties of the operations $\mat$ and $\mat^{-1}$, which are easily followed from the respective definitions. In addition, the existence of an outer inverse through full-rank decomposition and $M$-
QR decomposition has been studied under the $M$-product. 
\begin{proposition}
    Let $\mc{A} \in \mb{F}^{m\times n\times p}$,  $\mc{B} \in \mb{F}^{m\times n\times p}$, $\mc{C} \in \mb{F}^{n\times k\times p}$, and $M \in \mb{F}^{p\times p}$. Then 
\begin{enumerate}
 \item[(a)] $\mat(\alpha\mc{A}+\beta\mc{B}) = \alpha\mat(\mc{A}) +\beta\mat(\mc{B})$, where $\alpha,~\beta\in \mb{F}$.
        \item[(b)] $\mat(\mc{A})^* = \mat{(\mc{A}^*)}$.
             \item[ (c)] $\mat(\mc{A}*_M\mc{C}) = \mat(\mc{A}) \mat(\mc{C})$.
\end{enumerate}
\end{proposition}

\begin{proposition}
Let $M \in \mb{F}^{p\times p}$, $A\in \mb{F}_M^{mp\times np}$, $B\in \mb{F}_M^{mp\times np}$,  and $C\in \mb{F}_M^{np\times kp}$. Then 
\begin{enumerate}
 \item[(a)] $\mat^{-1}(\alpha A+\beta B) = \alpha\mat^{-1}({A}) +\beta\mat^{-1}({B})$,  where $\alpha,~\beta \in \mb{F}$. 
       \item[(b)] $\mat^{-1}(AC) = \mat^{-1}({A})*_M \mat^{-1}(C)$.
        \end{enumerate}
\end{proposition}
\begin{corollary}
Let $\mc{A}\in \mb{F}^{m\times m\times p}$ be an invertible tensor and $M \in \mb{F}^{p\times p}$. Then 
\begin{enumerate}
    \item[(a)] $\mat(\mc{A})^{-1} = \mat(\mc{A}^{-1})$.
    \item[(b)] $\mc{A}^{-1}=\mat^{-1}(\mat(\mc{A})^{-1})$.
    \item[(c)] $ \mat(\mc{A}) \mat(\mc{A}^{-1}) = {I}_{mp} = \mat(\mc{A})\mat(\mc{A})^{-1}$.
\end{enumerate}
\end{corollary}

On the other hand, we can also express $\{1\}$-inverse, the Moore-Penrose and Drazin inverse of $\mc{A}$ in terms of $\mat$ and $\mat^{-1}$ operations. 
\begin{proposition}
Let  $\mc{A} \in \mb{F}^{m\times n\times p}$ and $M\in\mb{F}^{p\times p}$. Then
\begin{enumerate}
    \item[(a)] $\mc{A}^-= \mat^{-1}[{\mat(\mc{A})}^-]$.
    \item[(b)] $\mc{A}^{\dagger} = \mat^{-1}[\mat(\mc{A})^{\dagger}]$.
     \item[(b)] $\mc{A}^{D} = \mat^{-1}[\mat(\mc{A})^{D}]$ if $m=n$ and $\ind(\mc{A})=k$.
\end{enumerate}

\begin{proof} 
(a) It follows from the below expression:
    \begin{eqnarray*}
\mc{A}*_M\mat^{-1}[\mat(\mc{A})^{(1)}]*_M\mc{A} &=
&\mat^{-1}(\mat(\mc{A}))*_M \mat^{-1}[\mat(\mc{A})^{(1)}]*_M \mat^{-1}(\mat(\mc{A}))\\ 
   & =& \mat^{-1}[\mat(\mc{A})(\mat(\mc{A}))^{(1)}\mat(\mc{A})]\\
    & =& \mat^{-1}[\mat(\mc{A})]= \mc{A}.
    \end{eqnarray*}
Similarly, we can show (b) and (c).
\end{proof}
\end{proposition}
The Moore-Penrose inverse and the Drazin inverse can be obtained from the outer inverse by specifying the range and null space, as discussed in the following results.
\begin{theorem}\label{thm-out}
    Let $\mc{A}\in\mb{F}^{m\times n\times p}$,  $\mc{B}\in\mb{F}^{m\times m\times p}$, $M\in\mb{F}^{p\times p}$, and $\ind(\mc{B})=k$. Then 
    \begin{enumerate}
        \item[(a)] $\mc{A}^{\dagger}=\mc{A}^{(2)}_{\rg(\mc{A}^*),\nl(\mc{A}^*)}$.
        \item[(b)] $\mc{B}^{D}=\mc{B}^{(2)}_{\rg(\mc{B}^k),\nl(\mc{B}^k)}$.
    \end{enumerate}
\end{theorem}
\begin{proof}
 (a) Let $\mc{X}=\mc{A}^\dagger$. Clearly $\mc{X}\in \mc{A}\{2\}$. Next,  we need to show $\rg(\mc{X})=\rg(\mc{A}^{*})$ and $\nl(\mc{X})=\nl(\mc{A}^*)$. From 
 \[\mat(\mc{A}^*)=\mat(\mc{A}^\dagger)\mat(\mc{A})\mat(\mc{A}^*)=\mat(\mc{X})\mat(\mc{A})\mat(\mc{A}^*)\]
 and  
 \[\mat(\mc{X})=\mat(\mc{A}^\dagger)=\mat(\mc{A}^*)\mat((\mc{A}^\dagger)^*)\mat(\mc{A}^{\dagger}),\]
 we get $R(\mat(\mc{A}^*))=R(\mat(\mc{X}))$. Thus by inverse operation, we have $\rg(\mc{X})=\rg(\mc{A}^{*})$.  If $Z\in N(\mat({A}^\dagger))$ then $\mat(\mc{A}^\dagger)Z=O$. Now
\[
   \mat(\mc{A}^*)Z=\mat(\mc{A}^*\m\mc{A}\m\mc{A}^\dagger)Z =\mat(\mc{A}^*\m\mc{A})\mat(A^{\dagger})Z=O.   
 \]
 Thus,  $N(\mat(\mc{A}^\dagger))\subseteq N(\mat(A^*)$. Conversely, if $Z\in N(\mat(\mc{A}^*))$ then $\mat(\mc{A}^*)Z=O$. In addition,
 \[\mat(\mc{A}^\dagger)Z=\mat(\mc{A}^\dagger\m(\mc{A}^\dagger)^*\m\mc{A}^*)Z=\mat(\mc{A}^\dagger\m(\mc{A}^\dagger)^*)\mat(\mc{A}^*)Z=O.\]
 Hence, we obtain $N(\mat{X})=N(\mat(\mc{A}^{*}))$. Again by inverse operation, we have $\nl(\mc{X})=\nl(\mc{A}^{*})$. \\
 (b) From $\mat(\mc{B})^D=\mat(\mc{B}^k)\mat((\mc{B}^D)^{k+1})\mbox{ and }\mat(\mc{B}^k)=\mat(\mc{B}^D)\mat(\mc{B}^{k+1})$, we obtain $R(\mat(\mc{B}^D))=R(\mat(\mc{B}^k))$ and hence  
$\rg(\mc{B}^D)=\rg(\mc{B}^k)$.  If $Z\in N(\mat(\mc{B}^D))$ then $\mat(\mc{B}^D)Z=O$. Now $\mat(\mc{B}^k)Z=\mat(\mc{B}^{k+1}\m\mc{B}^D)Z=\mat(\mc{B}^{k+1})\mat(\mc{B}^D)Z=O$ and subsequently, $N(\mat(\mc{B}^D))\subseteq N(\mat(\mc{B}^k))$. Conversely, if $Z\in N(\mat(\mc{B}^k))$ then $\mat(\mc{B}^k)Z=O$. Further, 
\[\mat(\mc{B}^D)Z=(\mat(\mc{B}^D)^{k+1}\mat(\mc{B}^k)Z=O.\]
Thus $N(\mat(\mc{B}^D))= N(\mat(\mc{B}^k))$. Applying inverse operation, we obtain $\nl(\mc{X})=\nl(\mc{B}^D)=\nl(\mc{B}^k)$.
 \end{proof}

\begin{lemma}\label{pfullrank}
Let $ \mc{A}\in\mathbb{F}^{m\times n\times p},~\mc{W}\in\mathbb{F}^{n\times m\times p},~ \mc{M}\in\mathbb{F}^{p\times p}$, and $\rank(\mc{A})=rp$. Assume that:
\begin{enumerate}
    \item[(a)] $\rg(\mc{W})$  is a subspace of $\mb{F}^{n\times1\times p}$ with dimension $sp\leq rp$.
    \item[(b)] $\nl(\mc{W})$ is a subspace of $\mb{F}^{m\times1\times p}$ with dimension $mp-sp$.
    \item[(c)] $\mc{W}=\mc{B}*_M\mc{C}$ is a full-rank decomposition of $\mc{W}$.
    \item[(d)] $\mc{A}^{(2)}_{\rg(\mc{W}),\nl(\mc{W})}$ exists.
\end{enumerate}
Then
 $\mc{C}*_M\mc{A}*_M\mc{B}$ is invertible  
   and $\mc{A}^{(2)}_{\rg(\mc{W}),\nl(\mc{W})}=\mc{B}\m(\mc{C}*_M\mc{A}*_M\mc{B})^{-1}*_M\mc{C}=\mc{A}^{(2)}_{\rg(\mc{B}),\nl(\mc{C})}$.
\end{lemma}
\begin{proof}
Let $A=\mat(\mc{A})$, $W=\mat(\mc{W})$, $B=\mat(\mc{B})$ and $C=\mat(\mc{C})$. Using the assumptions, we have  
\begin{itemize}
    \item $R(\mat(\mc{W}))$  is a subspace of $\mb{F}^{np}$ with dimension $sp\leq rp$.
    \item $N(\mat(\mc{W
    }))$ is a subspace of $\mb{F}^{mp}$ with dimension $mp-sp$.
    \item $\mat(\mc{W})=\mat(\mc{B})\mat(\mc{C})$ is a full-rank decomposition of $\mat(\mc{A})$.
    \item $\mat(\mc{A})^{(2)}_{R(\mat(\mc{W})),N(\mat(\mc{W}))}$ exists.
\end{itemize}
By using Theorem 3.1 \cite{sheng2007}, we obtain
$\mat(\mc{C})\mat(\mc{A})\mat(\mc{B})$ invertible and 
\begin{eqnarray*}\label{fl-rank}
\mat(\mc{A})^{(2)}_{R(\mat(\mc{W})),N(\mat(\mc{W}))}&=&\mat(\mc{B})(\mat(\mc{C})\mat(\mc{A})\mat(\mc{B}))^{-1}\mat(\mc{C})\\
&=&\mat(\mc{A})^{(2)}_{R(\mat(\mc{B})),N(\mat(\mc{C}))}. 
\end{eqnarray*}
By applying $\mat^{-1}$, we obtain  $\mc{C}*_M\mc{A}*_M\mc{B}$ is invertible and  
\begin{eqnarray*}
\mc{A}^{(2)}_{\rg(\mc{W}),\nl(\mc{W})}=\mc{B}*_M(\mc{C}*_M\mc{A}*_M\mc{B})^{-1}*_M\mc{C}=\mc{A}^{(2)}_{\rg(\mc{B}),\nl(\mc{C})}.
\end{eqnarray*}
\end{proof}

In the result below, we discuss the computation of the outer inverse by using QR decomposition based on an arbitrary $M$-product. It will be a more general form and serves as an extension of the earlier work on the tensors via $t$-product \cite{behera2023} and matrices by Chan \cite{Chan1987rank}.

\begin{theorem}\label{qrthm}
 Let $ \mc{A}\in\mathbb{F}^{m\times n\times p},~\mc{W}\in\mathbb{F}^{n\times m\times p},~ M\in\mathbb{F}^{p\times p}$,   $\rank(\mc{A})=rp$ and  $\rank(\mc{W})=sp$ with $sp\leq rp$. Consider the $\mc{Q}*_M\mc{R}$  decomposition of $\mc{W}$ be of the form
 \begin{equation}
     \mc{W}*_M\mc{P}=\mc{Q}*_M\mc{R},
 \end{equation}
 where $\mc{P}\in\mb{F}^{m\times m\times p}$ is a permutation tensor, $\mc{Q}\in \mb{F}^{n\times n\times p}$ satisfying $\mc{Q}*_M\mc{Q}^*=\mc{I}_{nnp}$, and $\mc{R}\in\mb{F}^{n\times m\times p}$ with $\rank(\mc{R})=sp$. The tensor $\mc{P}$ to be chosen so that it partitions $\mc{Q}$ and $\mc{R}$ in the following form
 \begin{equation}
     \mc{Q}=\begin{bmatrix}
      \tilde{\mc{Q}} & \mc{Q}_{12}   
     \end{bmatrix},~\mc{R}=\begin{bmatrix}
         \mc{R}_{11} & \mc{R}_{22}\\
         \mc{O} &\mc{O}
     \end{bmatrix}=\begin{bmatrix}
         \tilde{\mc{R}}\\
         \mc{O}
     \end{bmatrix},
 \end{equation}
 where $\tilde{\mc{Q}}\in\mb{F}^{n\times s\times p}$, $\mc{R}_{11}\in\mb{F}^{s\times s\times p}$ is nonsingular, and $\tilde{\mc{R}}\in\mb{F}^{s\times m\times p}$. If $\mc{A}^{(2)}_{\rg(\mc{W}),\nl(\mc{W})}$ exists then:
 \begin{enumerate}
\item[ (a)] $\tilde{\mc{R}}*_M\mc{P}^**_M\mc{A}*_M\tilde{\mc{Q}}$ is invertible.
\item[ (b)] $\mc{A}^{(2)}_{\rg(\mc{W}),\nl(\mc{W})}=\tilde{\mc{Q}}*_M(\tilde{\mc{R}}*_M\mc{P}^**_M\mc{A}*_M\tilde{\mc{Q}})^{-1}*_M\tilde{\mc{R}}*_M\mc{P}^*$.
\item[(c)] $\mc{A}^{(2)}_{\rg(\mc{W}),\nl(\mc{W})}=\mc{A}^{(2)}_{\rg(\tilde{\mc{Q}}),\nl(\tilde{\mc{R}}*_M\mc{P}^*)}$.
\item[\rm (d)] $\mc{A}^{(2)}_{\rg(\mc{W}),\nl(\mc{W})}=\tilde{\mc{Q}}*_M({\tilde{\mc{Q}}}^**_M\mc{W}*_M\mc{A}*_M\tilde{\mc{Q}})^{-1}*_M\tilde{\mc{Q}}^**_M\mc{W}$.
 \end{enumerate}
\end{theorem}
 \begin{proof}
 Let $ \mc{W}*_M\mc{P}=\mc{Q}*_M\mc{R}$. From the partition of $\mc{Q}$ and $\mc{R}$, we get $\mc{W}=\mc{Q}\m\mc{R}\m\mc{P}^*=\tilde{\mc{Q}}\m\tilde{\mc{R}}\m\mc{P}^*$ and consequently, $\mat(\mc{W})=\mat(\tilde{\mc{Q}})\mat(\tilde{\mc{R}}*\mc{P}^{*})$. It can be proved that $rank(\mat(\tilde{\mc{Q}}))=sp=rank(\mat(\tilde{\mc{R}}*_M\mc{P}^*))$. Thus  $\tilde{\mc{Q}}\m\tilde{\mc{R}}\m\mc{P}^*$ is full-rank decomposition of $\mc{W}$. Hence, by Lemma \ref{pfullrank}, we conclude the proof of parts (a), (b), and (c).\\
    (d) It follows by substituting $\tilde{\mc{R}}*_M\mc{P}^*={\tilde{\mc{Q}}}^*\m\mc{W}$ in part (b).
 \end{proof}
 {\small{
\begin{algorithm}[H]
 \caption{$M$-QR Decomposition} \label{alg:mqr}
\begin{algorithmic}[1]
  \Procedure{$M$-QR}{$\mc{A}$}
\State {\bf Input} Tensor $\mc{A} \in \mb{F}^{m \times n\times p}$ and an invertible matrix $M\in\mb{F}^{p\times p}$
\State Compute $\hat{\mc{A}}=\mc{A}\times_3 M$
\For{$i \gets 1$ to $p$} 
   \State $[\hat{\mc{Q}}(:,:,i), \hat{\mc{R}}(:,:,i), \hat{\mc{P}}(:,:,i)] = \texttt{qr}( \hat{\mc{A}}(:,:,i))$ 
   \EndFor
\State Compute $\mc{Q}=\hat{\mc{Q}}\times_3M^{-1},~\mc{R}=\hat{\mc{R}}\times_3M^{-1},~\mc{P}=\hat{\mc{P}}\times_3M^{-1}$
\State \Return $\mc{Q},~\mc{P},~\mc{R}$
\EndProcedure
 \end{algorithmic}
\end{algorithm}
}}

  {\small{
\begin{algorithm}[H]
  \caption{Computation of $\mc{A}^{(2)}_{\rg({W}), \nl({W})}$ by using $M$-QR decomposition} \label{AlgouterQR}
  \begin{algorithmic}[1]
 \Procedure{outerinv}{$\mc{A}$}
  \State  {\bf Input} $\mc{A} \in \mb{F}^{m \times n\times p}$, $\mc{W} \in \mb{F}^{n \times m\times p}$,  $M\in\mb{F}^{p\times p}$,  $det(M) \neq 0$, $\rank(\mc{W})=sp \leq rp = \rank(\mc{A})$
\State Compute $\mc{P}, \mc{Q}, \mc{R}$ using Algorithm \ref{alg:mqr}

\State $\mat(\mc{Q})$ and $\mat(\mc{R})$ will be in the following form
 \begin{equation*}
     \mat(\mc{Q})=\begin{bmatrix}
      \tilde{{Q}} & \bar{{Q}}   
     \end{bmatrix},~\mat({R})=\begin{bmatrix}
         \bar{R} & \bar{\bar{R}}\\
         O &O
     \end{bmatrix}=\begin{bmatrix}
         \tilde{{R}}\\
         O
     \end{bmatrix},
 \end{equation*}
 \hspace{.5cm} where $\tilde{Q}\in\mb{F}^{np\times sp}, \tilde{R}\in\mb{F}^{sp\times mp}$  and $\bar{{R}}\in\mb{F}^{sp\times sp}$ (nonsingular) are block matrices 
\State Compute $  \tilde{\mc{Q}}=\mat^{-1}(\tilde{{Q}})$ and $\tilde{\mc{R}}=\mat^{-1}(\tilde{R})$
\State Compute $\mc{Y} = \tilde{\mc{R}}*_M\mc{P}^**_M\mc{A}*_M\tilde{\mc{Q}}$ and $\mc{Z} = \tilde{\mc{Q}}^**_M\mc{W}*_M\mc{A}*_M\tilde{\mc{Q}}$
\State Compute $\mc{L}_{\mc{Y}}=\mc{Y}\times_3 M$ and $\mc{L}_{\mc{Z}}=\mc{Z}\times_3 M$
\For{$i \gets 1$ to $p$} 
   \State $\hat{\mc{Y}}(:,:,i) = (\mc{L}_{\mc{Y}}(:,:,i))^{-1}$,~~$\hat{\mc{Z}}(:,:,i) =(\mc{L}_{\mc{Z}}(:,:,i))^{-1}$  
   \EndFor
\State Compute $\mc{X}_1=\hat{\mc{Y}}\times_3M^{-1}$ and $\mc{X}_2=\hat{\mc{Z}}\times_3M^{-1}$
\State Compute
$\mc{A}^{(2)}_{\rg(\mc{W}),\nl(\mc{W})} =\tilde{\mc{Q}}*_M\mc{X}_1*_M\tilde{\mc{R}}*_M\mc{P}^*=\tilde{\mc{Q}}*_M\mc{X}_2*_M\tilde{\mc{Q}}^**_M\mc{W}$
   \State \textbf{return } $\mc{A}^{(2)}_{\rg({W}), \nl({W})}$ 
    \EndProcedure
  \end{algorithmic}
\end{algorithm}
}}
 
\begin{corollary}\label{MCor}
Let $ \mc{A}\in\mathbb{F}^{m\times n\times p},~\mc{W}\in\mathbb{F}^{n\times m\times p},~ \mc{M}\in\mathbb{F}^{p\times p}$,   $\rank(\mc{A})=rp$ and  $\rank(\mc{W})=sp$ with $sp\leq rp$. Let $\tilde{\mc{Q}}*(\tilde{\mc{R}}*\mc{P}^*)$ be the full-rank decomposition of $\mc{W}$, where $\mc{P}$, $\mc{Q}$, $\mc{R}$, $\tilde{\mc{Q}}$ and $\tilde{\mc{R}}$ as defined in Theorem \ref{qrthm}. Then 
$$\mc{A}^{(2)}_{\rg(\tilde{\mc{Q}}),\nl(\tilde{\mc{R}}*\mc{P}^*)}=\left\{
\begin{array}{ll}
\mc{A}^\dagger, & \mbox{ if }\mc{W}=\mc{A}^*.\\
\mc{A}^D, & \mbox{ if }\mc{W}=\mc{A}^k,\ k\geq ind(\mc{A}),\mbox{ and }\mc{A}\in\mb{F}^{m\times m\times p}.
\end{array}
\right.$$
 \end{corollary}

\section{Hyperpower iteration method}\label{sec:hpi}
In this section, we develop the hyperpower iteration method \cite{hyper-1,hyper-2} for computing outer inverses of third-order tensors in the framework of $M$-product with a specific emphasis on minimizing tensor product operations per iteration cycle. Following the matrix-based hyperpower method \cite{ma2022}, our proposed iteration method achieves superior computational efficiency. We provide a rigorous mathematical analysis that establishes the theoretical convergence properties of this high-order method, including detailed proofs of its asymptotic behavior and error bounds. Extensive numerical experiments show the efficiency of our algorithm, which significantly improves the performance, particularly for large-scale, sparse multilinear systems. The results confirm that our method maintains numerical stability while substantially reducing computational complexity, making it particularly suitable for image-deblurring applications. The standard form of the $p(\geq 2)$-th order hyperpower iteration (HPI$p$) is 
\begin{equation}\label{eq:itrhpip}
\mc{Z}_{j+1}=\mc{Z}_{j}\m\left(\mc{I}+\mc{R}_{j}+\mc{R}_{j}^{2}+\cdots+\mc{R}_{j}^{p-1}\right),
\end{equation}
where $\mc{R}_{j}=\mc{I}-\mc{A}\m\mc{Z}_{j},$ $j\in\{0,1,2,\ldots\}$.  Note that the iterative method 
\eqref{eq:itrhpip} requires $p$ number of tensor tensor products (TTP) per cycle. So the potential of HPI$p$ methods depends on the IEI \cite{traub1964} and the CEI \cite{traub1964}. If  $p$ is the order of convergence and $n$ is the number of TTP, then the IEI and the CEI are calculated as follows.
\[IEI=\dfrac{p}{n} \mbox{ and } CEI=p^{1/n}.\] Consider the $19^{\rm th}$-order by setting $p=19$ in \eqref{eq:itrhpip} that is,
\begin{equation}\label{eq:19hpi}
\mc{Z}_{j+1}=\mc{Z}_{j}\m\left(\mc{I}+\mc{R}_{j}+\mc{R}_{j}^{2}+\cdots+\mc{R}_{j}^{18}\right).
\end{equation}
 We denote this scheme as $M$-HPI19, which requires $19$ TTP at each iteration step of the scheme, and hence, the computational cost becomes high. Our target is to rewrite \eqref{eq:19hpi} such that the number of  TTP decreases while the order of convergence remains unaltered. Keeping this in mind, we can factorize \eqref{eq:19hpi} as 
\begin{equation}\label{eq:19hpife}
\mc{Z}_{j+1}=\mc{Z}_{j}\m \left(\mc{I}+\left(\mc{R}_{j}+\mc{R}_{j}^{2}\right)\m  \Lambda_{j}\right),
\end{equation}
where $\Lambda_{j}=\mc{I}+\mc{R}_{j}^{2}+\mc{R}_{j}^{4}+\mc{R}_{j}^{6}+\mc{R}_{j}^{8}+\mc{R}_{j}^{10}+\mc{R}_{j}^{12}+\mc{R}_{j}^{14}+\mc{R}_{j}^{16}$. 
Compared to the HPI method \eqref{eq:19hpi}, the factorized form 
\eqref{eq:19hpife} will have the same order of convergence, while the number of TTP is less. Further, we can simplify \eqref{eq:19hpife} by factoring $\Lambda_{j}$ in a recursive manner, which is described below.  
\begin{itemize}
    \item[Step I:] Rewrite the degree $16$ expression $\Lambda_{j}$ as product of two $8$ degree expressions as
    \begin{eqnarray*}
    \Lambda_{j}&= & \left(\mc{I}+\alpha_{1} \mc{R}_{j}^{2}+\alpha_{2} \mc{R}_{j}^{4}+\alpha_{3} \mc{R}_{j}^{6}+\mc{R}_{j}^{8}\right)  \m \left(\mc{I}+\beta_{1} \mc{R}_{j}^{2}+\beta_{2} \mc{R}_{j}^{4}+\beta_{3} \mc{R}_{j}^{6}+\mc{R}_{j}^{8}\right)+\left(\zeta_{1} \mc{R}_{j}^{2}+\zeta_{2} \mc{R}_{j}^{4}\right).
\end{eqnarray*}
 Comparing the coefficients of the equal powers of $\mc{R}_{j}$ on both sides, we obtain the following nonlinear system equations:
\begin{eqnarray} \nonumber\label{eq:svn}
&&\alpha_{1}  + \beta_{1}  + \zeta_{1} =1,\\\nonumber
&&\alpha_{2}  + \alpha_{1} \beta_{1}  + \beta_{2}  + \zeta_{2} =1,\\\nonumber
&&\alpha_{3}  +  \alpha_{2} \beta_{1}  + \alpha_{1} \beta_{2}  + \beta_{3} =1,\\
&&2  + \alpha_{3} \beta_{1}  + \alpha_{2} \beta_{2}  + \alpha_{1} \beta_{3} =1,\\ \nonumber
&&\alpha_{1}  + \beta_{1}  + \alpha_{3} \beta_{2}  + \alpha_{2} \beta_{3} =1,\\\nonumber
&&\alpha_{2}  + \beta_{2}  + \alpha_{3} \beta_{3} =1,\\ \nonumber
&& \alpha_{3}  + \beta_{3}=1.
\end{eqnarray}
Assuming $\alpha_{3}=\beta_{3}=\frac{1}{2}$ and then solving  \eqref{eq:svn}, we get
\[\alpha_{1}=\dfrac{5}{496}(31+\sqrt{93}), \alpha_{2}=\dfrac{1}{8}(3+\sqrt{93}), ~~\beta_{1}=-\dfrac{5}{496}(\sqrt{93}-31),\]
\[\beta_{2}=\dfrac{1}{8}(3-\sqrt{93}),~~ \zeta_{1}=\dfrac{3}{8},~~ \zeta_{2}=\dfrac{321}{1984}.\]
\item[Step II:] Let $\mc{V}_{j}=\mc{I}+\alpha_{1} \mc{R}_{j}^{2}+\alpha_{2} \mc{R}_{j}^{4}+\alpha_{3} \mc{R}_{j}^{6}+\mc{R}_{j}^{8}$ . Next, we factorize $\mc{V}_{j}$ as follows to reduce the number of tensor tensor products.
    \begin{equation*}
    \mc{V}_{j}=\left(1+\tau_{1} \mc{R}_{j}^{2}+\mc{R}_{j}^{4}\right)\m \left(1+\tau_{2} \mc{R}_{j}^{2}+\mc{R}_{j}^{4}\right)+\tau_3 \mc{R}_{j}^{2}.
\end{equation*}
On comparing the similar power of $\mc{R}_{j}$ on both sides, we obtain
\begin{equation}\label{eq:gma}
   \tau_{1} + \tau_{2}  + \tau_{3} =\alpha_{1},~ 2  + \tau_{1} \tau_{2} =\alpha_{2},~\tau_{1}  + \tau_{2} =\alpha_{3}.
   \end{equation}
On solving \eqref{eq:gma} for $\tau_{1},~\tau_{2}$ and $\tau_{3}$, yields
\[\tau_{1}=\frac{1}{4}\left(\sqrt{27-2 \sqrt{93}}+1\right),~ \tau_{2}=\frac{1}{4}\left(1-\sqrt{27-2 \sqrt{93}}\right), ~
    \tau_3=\frac{1}{496}\left(5 \sqrt{93}-93\right).
\]
    \item[Step III:] Assume  $\mc{W}_{j}=\mc{I}+\beta_{1} \mc{R}_{j}^{2}+\beta_{2} \mc{R}_{j}^{4}+\beta_{3} \mc{R}_{j}^{6}+\mc{R}_{j}^{8}$.  In addition, we factorize $\mc{W}_{j}$ as given below.
\[\mc{W}_{j}=\left(1+\tau_{1} \mc{R}_{j}^{2}+\mc{R}_{j}^{4}\right)\m \left(1+\tau_{2} \mc{R}_{j}^{2}+\mc{R}_{j}^{4}\right)+\xi_{1} \mc{R}_{j}^{2}+\xi_{2} \mc{R}_{j}^{4}.\]
Again comparing the similar power of $\mc{R}_{j}$ on both sides, we have
\begin{equation}\label{eq:beta}
 \tau_{1} + \tau_{2}  + \xi_{1}=\beta_{1}, ~
2+ \tau_{1} \tau_{2}  + \xi_{2}=\beta_{2},~
\tau_{1} + \tau_{2}=\beta_{3}.
   \end{equation}
Solving Equation \eqref{eq:beta} for $\xi_{1}$ and $\xi_{2}$, resulted $\xi_{1}=\frac{1}{496}(-93-5 \sqrt{93})$, and $ \xi_{2}=-\frac{\sqrt{93}}{4}$.
\end{itemize}
The advantage of the above factorization is that we require only seven tensor tensor products per loop instead of $19$ with $19^{th}$-order of convergence. We summarize the $M$-HPI19 method as follows.
\begin{eqnarray} \label{eq:tenhpi19}
\nonumber
&&\mc{R}_j =\mc{I}-\mc{A}\m\mc{Z}_j,\\
\nonumber
&&\mc{U}_j=\left(\mc{I}+\tau_{1}\mc{R}^{2}_j+\mc{R}^{4}_j\right)\m\left(\mc{I}+\tau_{2}\mc{R}^{2}_j+\mc{R}^{4}_j\right),\\ 
&&\mc{V}_j=\mc{U}_j+\tau_3 \mc{R}^{2}_j,\\ 
\nonumber
&&\mc{W}_j=\mc{U}_j+\xi_{1} \mc{R}^{2}_j+\xi_{2} \mc{R}^{4}_j,\\ 
\nonumber
&&\mc{Z}_{j+1}=\mc{Z}_j\m \left(\mc{I}+\left(\mc{R}_j+\mc{R}^{2}_j\right)\m\left(\mc{V}_j\m \mc{W}_j+\zeta_{1} \mc{R}^{2}_j+\zeta_{2} \mc{R}^{4}_j\right)\right).
\end{eqnarray}
In the case of $p=9$, we denote it $M$-HPI9, which requires five TTP, and the method can be formulated as follows. 
\begin{eqnarray*} 
&&\mc{R}_j =\mc{I}-\mc{A}\m\mc{Z}_j,\\
&&\mc{U}_j=\frac{7}{8}\mc{R}_j+\mc{R}^{2}_j\m\left(\frac{1}{2}\mc{R}_j+\mc{R}^{2}_j\right),\\
&&\mc{V}_j=\frac{11}{16}\mc{I}-\frac{9}{8}\mc{R}_j+\frac{3}{4}\mc{R}^{2}_j+\mc{U}_j,\\ 
&&\mc{Z}_{j+1}=\mc{Z}_j\m \left(\mc{I}+\frac{51}{128}\mc{R}_j+\frac{39}{32}\mc{R}^{2}_j+\mc{U}_j\m\mc{V}_j\right).
\end{eqnarray*}
The comparison of CEI and IEI for different order convergence is presented in Figure \ref{fig:Fig3}. It is also observed that the CEI (1.552) for $M$-HPI9 is slightly better than the CEI (1.523) of $M$-HPI19. However, the IEI of $M$-HPI19 is much better than the IEI of $M$-HPI9. So  $M$-HPI19 keeps the balance between both efficiencies with less computational burden.  
\begin{figure}[H]
\centering
\subfigure{\includegraphics[scale=0.54]{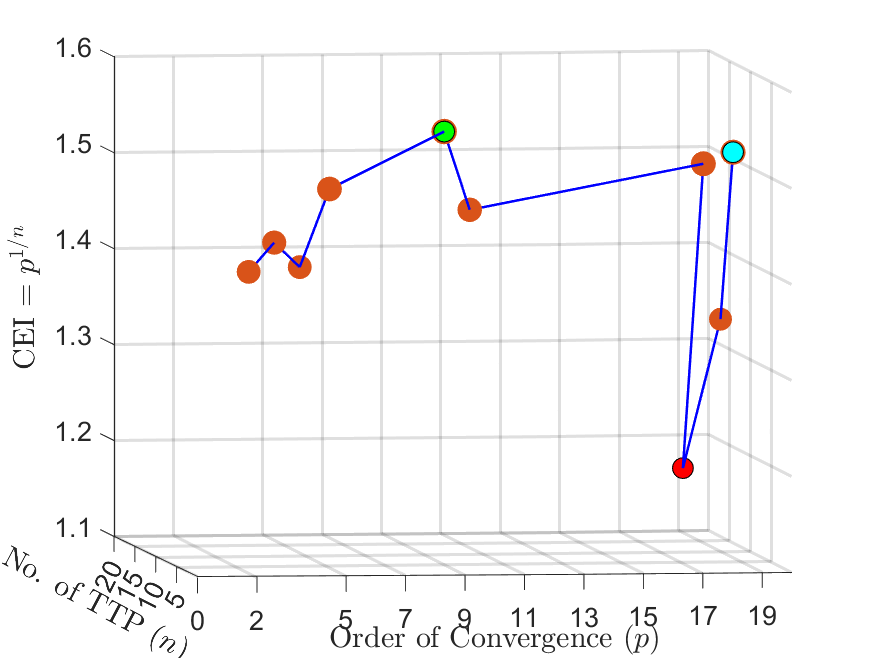}}
\subfigure{\includegraphics[scale=0.54]{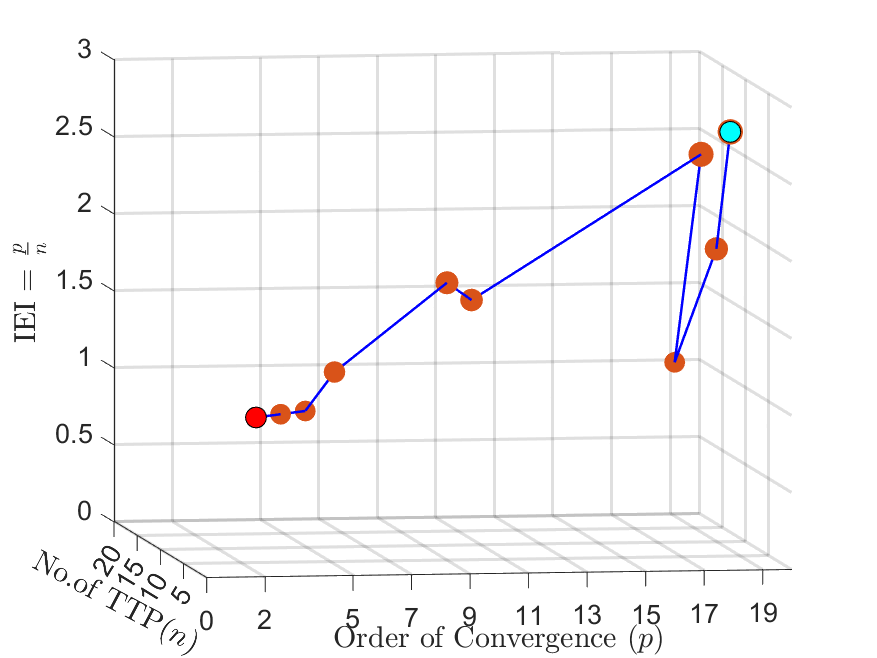}}
\caption{Comparison of CEI and IEI for different orders of convergence.}
\label{fig:Fig3}
\end{figure}
Next, we discuss the convergence of $M$-HPI19 under a suitable initial approximation. More details on choosing the initial approximation can be found in \cite{ing1,ing2}.
\begin{lemma}\label{lem-hpi19}
  Let $\mc{A} \in \mb{F}^{m\times n\times p}$ and $\mc{W} \in  \mb{F}^{n\times m\times p}$ such that $\rank(\mc{W} * \mc{A})=\rank(\mc{W})\leq \rank(\mc{A})$. Assume that $\mc{X}=\mc{A}^{(2)}_{\rg(\mc{W}),\nl(\mc{W})}$. Then 
       \begin{center}
        $\mc{Z}\m\mc{A}\m\mc{X}=\mc{Z}$ and  $\mc{X}\m\mc{A}\m\mc{Z}=\mc{Z}$ if and only $\rg(\mc{Z})\subseteq \rg(\mc{W})$ and $\nl(\mc{W})\subseteq \nl(\mc{Z})$.   
       \end{center}   
\end{lemma}
\begin{proof}
    Let $\mc{X}=\mc{A}^{(2)}_{\rg(\mc{W}),\nl(\mc{W})}$,  $\mc{Z}\m\mc{A}\m\mc{X}=\mc{Z}$ and  $\mc{X}\m\mc{A}\m\mc{Z}=\mc{Z}$. Then 
    \[\mat(\mc{Z})\mat(\mc{A})\mat(\mc{X})=\mat(\mc{Z})\mbox{ and }\mat(\mc{X})\mat(\mc{A})\mat(\mc{Z})=\mat(\mc{Z}).\]
    By Lemma \ref{lem2.15-rn}, we obtain $R(\mat(\mc{Z}))\subseteq R(\mat(\mc{W}))$ and $N(\mat(\mc{W}))\subseteq N(\mat(\mc{Z}))$. Now if $\mc{T}\in \rg(\mc{Z})$ then $\mc{T}=\mat^{-1}(Y)$, where $Y\in R(\mat(\mc{Z}))\subseteq R(\mat(\mc{W}))$. Therefore, $\mc{T}\in \rg(\mc{W})$ and subsequently, $\rg(\mc{Z})\subseteq \rg(\mc{W})$. Similarly, we can show the null condition $\nl(\mc{W})\subseteq \rg(\mc{Z})$.
    
    Conversely, let $\rg(\mc{Z})\subseteq \rg(\mc{W})$ and $\nl(\mc{W})\subseteq \nl(\mc{Z})$. If $Y\in R(\mat(\mc{Z}))$. Then $\mat^{-1}(Y)\in \rg(\mc{Z})\subseteq \rg(\mc{W})$. Thus $Y\in R(\mat(\mc{W}))$ and hence $R(\mat(\mc{Z}))\subseteq R(\mat(\mc{W}))$. Similarly, we can show $N(\mat(\mc{W}))\subseteq N(\mat(\mc{Z}))$. Applying Lemma \ref{lem2.15-rn}, we have  
    \begin{equation}\label{mat-eq}
\mat(\mc{Z})\mat(\mc{A})\mat(\mc{X})=\mat(\mc{Z})\mbox{ and }\mat(\mc{X})\mat(\mc{A})\mat(\mc{Z})=\mat(\mc{Z}).    
    \end{equation}
Operating $\mat^{-1}$ to the equation \eqref{mat-eq}, we obtain 
$\mc{Z}\m\mc{A}\m\mc{X}=\mc{Z}$ and  $\mc{X}\m\mc{A}\m\mc{Z}=\mc{Z}$.
\end{proof}
\begin{theorem}\label{thm:tenhpi19}
    Let $\mc{A} \in \mb{F}^{m\times n\times p}$ and $\mc{W} \in  \mb{F}^{n\times m\times p}$ such that $\rank(\mc{W} * \mc{A})=\rank(\mc{W})\leq \rank(\mc{A})$. Let $\mc{T}=\rg(\mc{W})$ and $\mc{S}=\nl(\mc{W})$. Then the $M$-HPI19,   $\left\{\mc{Z}_{j}\right\}_{j=0}^{\infty}$ generated by equation  \eqref{eq:tenhpi19} converges to $\mc{A}_{\mc{T}, \mc{S}}^{(2)}$ with the $19^{th}$-order of convergence only if the initial approximation $\mc{Z}_0=\gamma \mc{W}$ satisfies
\begin{equation*}
    \left\|\mc{A} \m \mc{A}_{\mc{T}, \mc{S}}^{(2)}-\mc{A}\m \mc{Z}_{0}\right\|_F<1.
\end{equation*}
\end{theorem}
\begin{proof}
Let $\Upsilon_{j}=\mc{A} \m  \mc{A}_{\mc{T}, \mc{S}}^{(2)}-\mc{A} \m \mc{Z}_{j}$. Then 
\begin{eqnarray*}
\Upsilon_{j+1}& =&\mc{A} \m  \mc{A}_{\mc{T},\mc{S}}^{(2)}-\mc{A}\m \mc{Z}_{j+1}\\
& =&\mc{A} \m  \mc{A}_{\mc{T},\mc{S}}^{(2)}-\mc{A}\m \mc{Z}_{j}\m \left(\mc{I}+(\mc{R}_{j}+\mc{R}_{j}^{2})\m (\mc{V}_{j}\m \mc{W}_{j}+\zeta_{1}\mc{R}_{j}^{2}+\zeta_{2}\mc{R}_{j}^{4})\right)\\
& =&\mc{A} \m  \mc{A}_{\mc{T},\mc{S}}^{(2)}\\
&&-\mc{A}\m \mc{Z}_{j}\m \left((\mc{I}+(\mc{R}_{j}+\mc{R}_{j}^{2})\m (\mc{I}+\mc{R}_{j}^{2}+\mc{R}_{j}^{4}+\mc{R}_{j}^{6}+\mc{R}_{j}^{8}+\mc{R}_{j}^{10}+\mc{R}_{j}^{12}+\mc{R}_{j}^{14}+\mc{R}_{j}^{16})\right) \\
& =&\mc{A} \m  \mc{A}_{\mc{T},\mc{S}}^{(2)}-\left(\mc{I}-\mc{R}_{j}\right)\m \left(\mc{I}+\mc{R}_{j}+\mc{R}_{j}^{2}+\mc{R}_{j}^{3}+\cdots+\mc{R}_{j}^{18}\right) \nonumber\\
& =&\mc{A}\m  \mc{A}_{\mc{T},\mc{S}}^{(2)}-\mc{I}+\mc{R}_{j}^{19} \nonumber\\
& =&\mc{A} \m  \mc{A}_{\mc{T},\mc{S}}^{(2)}-\mc{I}+\left(\mc{I}-\mc{A} \m \mc{Z}_{j}\right)^{19} \nonumber\\
& =&\mc{A}\m  \mc{A}_{\mc{T},\mc{S}}^{(2)}-\mc{I}+\left[\left(\mc{I}-\mc{A} \m  \mc{A}_{\mc{T},\mc{S}}^{(2)}\right)+\Upsilon_{j}\right]^{19}.
\end{eqnarray*}
From the definition of $\mc{A}_{\mc{T}, S}^{(2)}$, we have  
\begin{equation}\label{eqn-12}
\left(\mc{I}-\mc{A} \m  \mc{A}_{\mc{T}, S}^{(2)}\right)^{n}=\mc{I}-\mc{A} \m  \mc{A}_{\mc{T},\mc{S}}^{(2)},~n\in\mathbb{N}.    \end{equation}
Since $\mc{Z}_0=\gamma \mc{W}$, it can be easily shown that 
\[\rg(\mc{Z}_j)\subseteq \rg(\mc{W}),~  \nl(\mc{W})\subseteq \nl(\mc{Z}_j), ~j\in\mathbb{N}.\]
Thus by Lemma \ref{lem-hpi19}, we obtain  $\mc{A}_{\mc{T}, \mc{S}}^{(2)}\m\mc{A} \m \mc{Z}_{j}=\mc{Z}_{j}$.
Now 
\begin{equation}\label{eqq-13}
    \left(\mc{I}-\mc{A} \m  \mc{A}_{\mc{T}, S}^{(2)}\right)\m \Upsilon_{j}=\mc{A} \m \mc{Z}_{j}-\mc{A}\m\mc{A}_{\mc{T}, \mc{S}}^{(2)}\m\mc{A} \m \mc{Z}_{j}=\mc{A} \m \mc{Z}_{j}-\mc{A} \m \mc{Z}_{j}=\mc{O}.
\end{equation}

Substituting equations \eqref{eqn-12} and \eqref{eqq-13} in $\Upsilon_{j+1}$, we get
\begin{equation*}
    \Upsilon_{j+1}=\mc{A} \m  \mc{A}_{\mc{T}, S}^{(2)}-\mc{I}+\mc{I}-\mc{A} \m  \mc{A}_{\mc{T}, S}^{(2)}+\Upsilon_{j}^{19}=\Upsilon_{j}^{19}.
\end{equation*}

Let $\Psi_{j}=\mc{A}_{\mc{T}, \mc{S}}^{(2)}-\mc{Z}_{j}$. Then 
$\mc{A}\m\Psi_{j+1}=\Upsilon_{j+1}=\Upsilon_{j}^{19}=\left(\mc{A}\m\Psi_{j}\right)^{19}$,  and by Lemma \ref{lem-hpi19} we have 
\begin{eqnarray*}
\|\Psi_{j+1}\|_F 
& =&\|\mc{A}_{\mc{T}, \mc{S}}^{(2)}-\mc{Z}_{j+1}\|_F=\|\mc{A}_{\mc{T}, \mc{S}}^{(2)}\m \mc{A} \m \left(\mc{A}_{\mc{T}, \mc{S}}^{(2)}-\mc{Z}_{j+1}\right)\|_F=\|\mc{A}_{\mc{T}, \mc{S}}^{(2)}\m\mc{A}\m\Psi_{j+1}\|_F\\
&=&\|\mc{A}_{\mc{T}, \mc{S}}^{(2)}\m(\mc{A}\m\Psi_{j})^{19}\|_F\leq \|\mc{A}_{\mc{T}, \mc{S}}^{(2)}\|_F\|\mc{A}\|_F^{19}\|\Psi_{j}\|_F^{19},
\end{eqnarray*}
and 
\[
\left\|\Psi_{j+1}\right\|_F=\left\|\mc{A}_{\mc{T}, \mc{S}}^{(2)}\m \Upsilon_j^{19}\right\|_F\leq\left\|\mc{A}_{\mc{T}, \mc{S}}^{(2)}\right\|_F\left\|\Upsilon_j\right\|_F^{19} \leq\left\|\mc{A}_{\mc{T}, \mc{S}}^{(2)}\right\|_F\left\| \Upsilon_{j-1}\right\|_F^{19^2} \leq \cdots \leq\left\| \mc{A}_{\mc{T}, \mc{S}}^{(2)}\right\|_F\left\|\Upsilon_0\right\|_F^{19^{j+1}}.  
\]
 Since $\|\Upsilon_{0}\|_F<1$, so $\Psi_{j+1}\to 0$ as $j\to\infty$ and completes the proof. 
\end{proof}

The following corollary immediately follows from Theorem \ref{thm-out} and Theorem \ref{alg:MHPI}.
\begin{corollary}\label{cor-4-2}
    Let the condition Theorem \ref{thm:tenhpi19} be satisfied and the sequence $\{\mc{Z}_{j}\}$ obtained from equation \eqref{eq:tenhpi19}. Then 
\begin{equation*}
    \lim _{j \to \infty} \mc{Z}_{j} = 
    \begin{cases}
    \mc{A}^{\dagger}, & \mc{Z}_{0}=\alpha \mc{A}^{*}.\\ 
    \mc{A}^{D}, & \mc{Z}_{0}=\alpha \mc{A}^{k},~ k\geq ind(\mc{A}).
    \end{cases}
\end{equation*}
\end{corollary}
{\small{
\begin{algorithm}[H]
 \caption{Computation of outer inverse by $M$-HPI19} \label{alg:MHPI}
\begin{algorithmic}[1]
\Procedure{$M$-HPI19}{$\mc{A},\mc{Z}_0,\epsilon$}
\State {\bf Input} $\mc{A} \in \mb{F}^{m \times n\times p}$ and an invertible matrix $M\in\mb{F}^{p\times p}$, Tolerance $\epsilon$
\State {\bf Choose} $\mc{W}$, where $\mc{W}\in \mb{R}^{n \times m\times p}$ such that $\rank(\mc{W}\m \mc{A})=\rank(\mc{W})\leq \rank(\mc{A})$
\State Initial guess $\mc{Z}_0$
\While{ (true)} 
\State $\mc{R}=\mc{I}-\mc{A}\m  \mc{Z}_0$
\State $\mc{R}^{2}=\mc{R}\m \mc{R}$
\State $\mc{R}^{4}=\mc{R}^{2}\m \mc{R}^{2}$
\State $\mc{U}=\left(\mc{I}+\tau_{1}\mc{R}^{2}+\mc{R}^{4}\right)\m\left(\mc{I}+\tau_{2}\mc{R}^{2}+\mc{R}^{4}\right)$ 
\State $\mc{V}=\mc{U}+\tau_3 \mc{R}^{2}$ 
\State $\mc{W}=\mc{U}+\xi_{1} \mc{R}^{2}+\xi_{2} \mc{R}^{4}$
\State $\mc{Z}=\mc{Z}_0\m \left(\mc{I}+\left(\mc{R}+\mc{R}^{2}\right)\m\left(\mc{V}\m \mc{W}+\zeta_{1} \mc{R}^{2}+\zeta_{2} \mc{R}^{4}\right)\right)$
\If{($\|\mc{Z}-\mc{Z}_0\|_F< \epsilon$)}
\Break
\EndIf
\State $\mc{Z}_0=\mc{Z}$
\EndWhile
\State \Return $\mc{Z}$
\EndProcedure
 \end{algorithmic}
\end{algorithm}
}}
\section{Numerical examples}
In this manuscript, all the numerical examples are computed using MATLAB, R2023b, which runs on a Mac Pro with a CPU [12-Core Intel Xeon W   3.3 GHz], 96GB RAM, and  Ventura OS.

 \subsection{Examples based on $M$-QR decomposition}
The following notations are used for errors associated with different matrix and tensor equations. In these notations, $M$ can be specified as a DFT matrix, $M_1$ matrix, and any random invertible matrix. For example, $\mc{E}_1^{DFT}$ means the norm evaluated when $M$ is a DFT matrix. Whenever we write in terms of $M$, it is understood that the matrix $M$ is taken randomly.

{\small{
\begin{center}
{\small{Errors associated with matrix and tensor equations.}}
    \renewcommand{\arraystretch}{1.2}
    \begin{tabular}{lll}
    \hline
     $\mc{E}^{M}_{1} = \|\mc{A}-\mc{A}*_M\mc{X}*_M\mc{A}\|_F$  & $\mc{E}^M_{2} = \|\mc{X}-\mc{X}*_M\mc{A}*_M\mc{X}\|_F$ & $\mc{E}^M_{3} =\|\mc{A}*_M\mc{X}-(\mc{A}*_M\mc{X})^T\|_F$ \\  
    $\mc{E}^M_{4} =\|\mc{X}*_M\mc{A}-(\mc{X}*_M\mc{A})^T\|_F$&$\mc{E}^{M}_{5} = \|\mc{A}*_M\mc{X}-\mc{X}*_M\mc{A}\|_F$
       &$\mc{E}^{M}_{1^k} = \|\mc{X}*_M\mc{A}^{k+1} -\mc{A}^k\|_F$\\
       ${E}_{1} = \|{A}-{A}{X}{A}\|_F$  &${E}_{2} = \|{X}-{X}{A}{X}\|_F$&${E}_{3} =\|{A}{X}-({A}{X})^{T}\|_F$ \\
       ${E}_{4} =\|{X}{A}-({X}{A})^T\|_F$&${E}_{5} = \|{A}{X}-XA\|_F$
       &${E}_{1^k} = \|XA^{k+1} -{A}^k\|_F$ \\
     \hline
    \end{tabular}     
    \end{center}
}}
Next, we discuss the comparison analysis regarding mean CPU time (MT$^M$) and errors (Error$^M$) associated with computing the Moore-Penrose inverse by $M$-QR decomposition. In the first test tensor, we have considered the frontal slices as the $chow$ matrices of order $n$. For different choices of $n$ and $M$,  MT$^M$, and Error$^M$  are provided in Table \ref{tab:error-chow}. In addition, the mean CPU time against the order for the same test tensor is provided in Figure \ref{fig:comp-MQR} (a). Another test tensor was generated by considering the frontal slices as the $cycol$ matrix of order $n$, for which the mean CPU time against the order is plotted in Figure \ref{fig:comp-MQR} (b).

  \begin{table}[H]
    \begin{center}
       \caption{ Comparison analysis for computing the Moore-Penrose inverse for different choices of $M$.}
       \vspace{0.2cm}
          \renewcommand{\arraystretch}{1.2}
    \begin{tabular}{ccccccc}
    \hline
        Size of $\mc{A}$  & MT$^{DFT}$  & MT$^{M_1}$ & MT$^M$ & Error$^{DFT}$ & Error$^{M_1}$ &Error$^M$\\ 
           \hline
    \multirow{4}{*}{$150\times 150\times 150$} &  \multirow{4}{*}{1.83} &\multirow{4}{*}{1.01}&\multirow{4}{*}{1.07}& 
    $\mc{E}^{DFT}_{1}= 1.76e^{-11}$  & $\mc{E}^{M_1}_{1}=1.26e^{-08}$ & $\mc{E}^{M}_{1}=1.14e^{-11}$
    \\ 
& &  & & $\mc{E}^{DFT}_{2}= 1.30e^{-16}$  & $\mc{E}^{M_1}_{2}=8.69e^{-13}$ & $\mc{E}^{M}_{2}=3.00e^{-10}$
    \\  
          &  & & & $\mc{E}^{DFT}_{3}= 2.13e^{-14}$  & $\mc{E}^{M_1}_{3}=2.10e^{-11}$ & $\mc{E}^{M}_{3}=1.18e^{-10}$
    \\ 
     &  & & &$\mc{E}^{DFT}_{4}= 1.59e^{-13}$  & $\mc{E}^{M_1}_{4}=2.12e^{-11}$ & $\mc{E}^{M}_{4}=8.94e^{-10}$
    \\ 
    \hline
    \multirow{4}{*}{$350\times 350\times 350$} &  \multirow{4}{*}{13.75} &\multirow{4}{*}{7.91}&\multirow{4}{*}{7.86}& 
    $\mc{E}^{DFT}_{1}= 9.95e^{-11}$  & $\mc{E}^{M_1}_{1}=3.34e^{-07}$ & $\mc{E}^{M}_{1}=8.15e^{-11}$
    \\ 
& &  & & $\mc{E}^{DFT}_{2}= 1.65e^{-16}$  & $\mc{E}^{M_1}_{2}=2.99e^{-12}$ & $\mc{E}^{M}_{2}=3.55e^{-07}$
    \\  
          &  & & & $\mc{E}^{DFT}_{3}= 5.31e^{-14}$  & $\mc{E}^{M_1}_{3}=5.78e^{-11}$ & $\mc{E}^{M}_{3}=6.66e^{-10}$
    \\ 
     &  & & & $\mc{E}^{DFT}_{4}= 3.80e^{-13}$  & $\mc{E}^{M_1}_{4}=8.93e^{-11}$ & $\mc{E}^{M}_{4}=8.33e^{-09}$
    \\ 
    \hline
        \multirow{4}{*}{$450\times 450\times 450$} &  \multirow{4}{*}{61.29} &\multirow{4}{*}{32.33}&\multirow{4}{*}{32.22}& 
    $\mc{E}^{DFT}_{1}= 2.51e^{-10}$  & $\mc{E}^{M_1}_{1}=1.13e^{-07}$ & $\mc{E}^{M}_{1}=2.29e^{-10}$
    \\ 
& &  & & $\mc{E}^{DFT}_{2}= 1.96e^{-16}$  & $\mc{E}^{M_1}_{2}=7.20e^{-12}$ & $\mc{E}^{M}_{2}=1.71e^{-09}$
    \\  
          &  & & & $\mc{E}^{DFT}_{3}= 8.27e^{-14}$  & $\mc{E}^{M_1}_{3}=1.56e^{-10}$ & $\mc{E}^{M}_{3}=3.01e^{-10}$
    \\ 
     &  & & & $\mc{E}^{DFT}_{4}= 9.26e^{-12}$  & $\mc{E}^{M_1}_{4}=2.13e^{-10}$ & $\mc{E}^{M}_{4}=4.12e^{-10}$
    \\ 
    \hline
    \end{tabular}
       \label{tab:error-chow}  
    \end{center}
\end{table}

\begin{figure}[H]
\centering
\subfigure[Frontal slices are taken as $chow$ matrix of order $n$]{\includegraphics[scale=0.43]{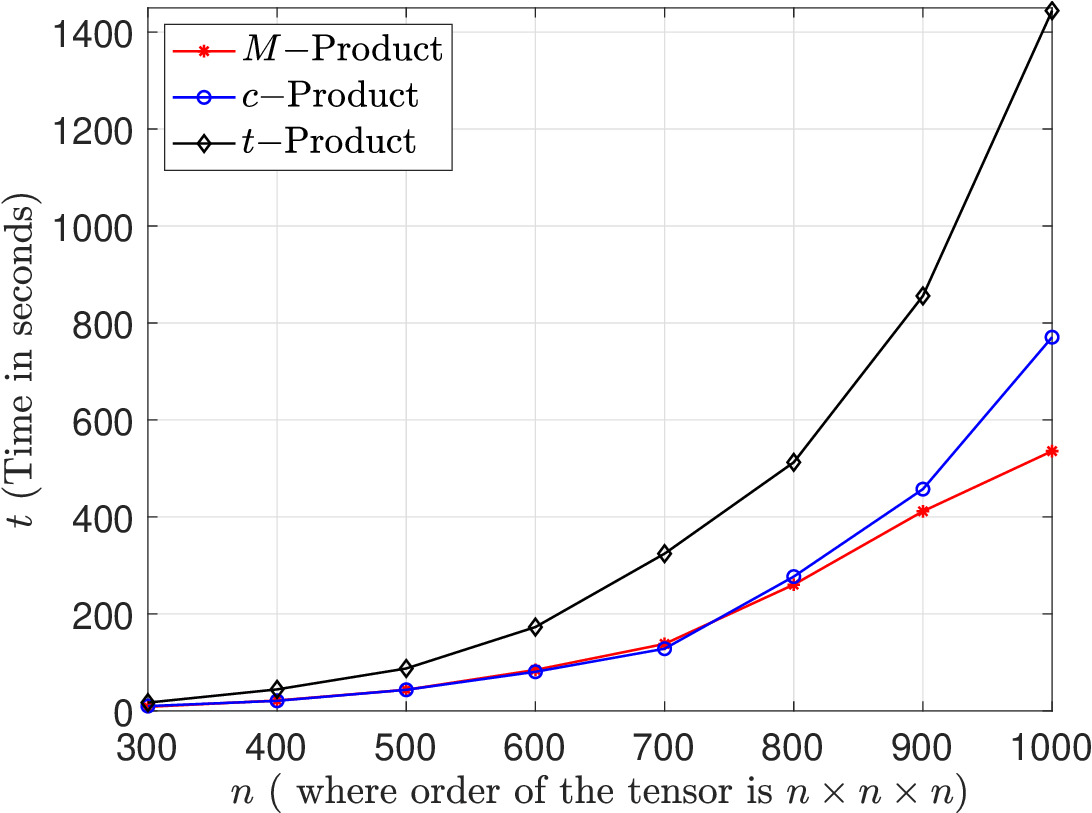}}
\subfigure[Frontal slices are taken as $cycol$ matrix of order $n$]{\includegraphics[scale=0.43]{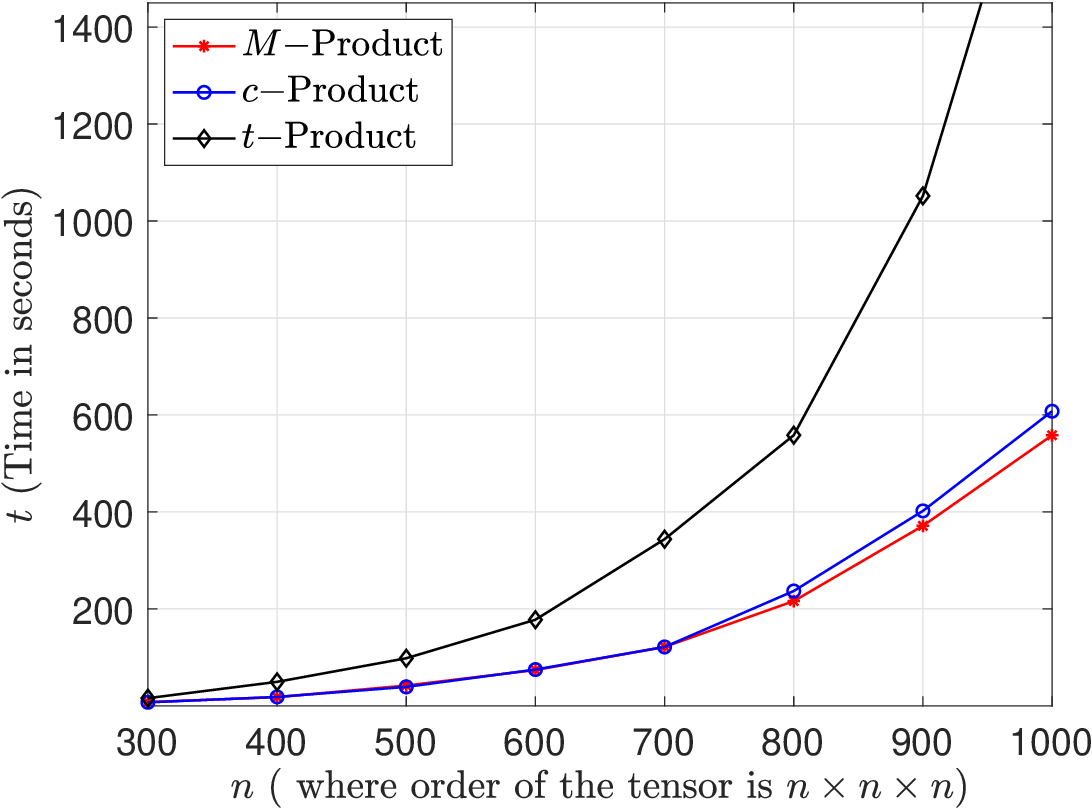}}
\caption{Comparison analysis of mean CPU time and order of tensor for computing the Moore--Penrose inverse.}
\label{fig:comp-MQR}
\end{figure}

The comparison analysis of mean CPU time (MT$^M$) and errors (Error$^M$) associated while computing the Drazin inverse by $M$-QR decomposition is presented in Table \ref{tab:error-gearmat} and Figure \ref{fig:comp-MDR}. For different choices of $n$ and $M$, and taking the frontal slices as $gearmat$, we have calculated MT$^M$ and Error$^M$, which are provided in Table \ref{tab:error-gearmat}. Further, the mean CPU time against the order for the same test tensor is provided in Figure \ref{fig:comp-MDR} (a). In addition, a few index $1$ tensors were randomly generated, and a comparison of the mean CPU times against the tensor orders is plotted in Figure \ref{fig:comp-MDR} (b).
\begin{table}[H]
    \begin{center}
       \caption{Comparison analysis for computing the Drazin inverse for different choices of $M$.}
        \vspace{0.2cm}
        \renewcommand{\arraystretch}{1.2}
    \begin{tabular}{ccccccc}
    \hline
        Size of $\mc{A}$  & MT$^{DFT}$  & MT$^{M_1}$ & MT$^M$ & Error$^{DFT}$ & Error$^{M_1}$ &Error$^M$\\
           \hline
    \multirow{4}{*}{$150\times 150\times 150$} &  \multirow{4}{*}{1.26} &\multirow{4}{*}{0.93}&\multirow{4}{*}{1.00}& 
    $\mc{E}^{DFT}_{1^2}= 3.23e^{-08}$  & $\mc{E}^{M_1}_{1^2}=1.33e^{-05}$ & $\mc{E}^{M}_{1^2}=8.10e^{-11}$
    \\ 
& &  & & $\mc{E}^{DFT}_{2}= 9.72e^{-10}$  & $\mc{E}^{M_1}_{2}=1.53e^{-10}$ & $\mc{E}^{M}_{2}=4.0e^{-13}$
    \\  
          &  & & & $\mc{E}^{DFT}_{5}= 2.26e^{-12}$  & $\mc{E}^{M_1}_{5}=7.06e^{-10}$ & $\mc{E}^{M}_{5}=1.55e^{-10}$
        \\ 
    \hline
    \multirow{4}{*}{$300\times 300\times 300$} &  \multirow{4}{*}{9.91} &\multirow{4}{*}{7.02}&\multirow{4}{*}{6.82}& 
    $\mc{E}^{DFT}_{1^2}= 8.99e^{-07}$  & $\mc{E}^{M_1}_{1^2}=6.65e^{-04}$ & $\mc{E}^{M}_{1^2}=4.72e^{-10}$
    \\ 
& &  & & $\mc{E}^{DFT}_{2}= 6.59e^{-13}$  & $\mc{E}^{M_1}_{2}=2.81e^{-09}$ & $\mc{E}^{M}_{2}=4.91e^{-09}$
    \\  
          &  & & & $\mc{E}^{DFT}_{5}= 4.91e^{-11}$  & $\mc{E}^{M_1}_{5}=8.68e^{-09}$ & $\mc{E}^{M}_{5}=5.90e^{-09}$
        \\ 
    \hline
       \multirow{4}{*}{$450\times 450\times 450$} &  \multirow{4}{*}{46.21} &\multirow{4}{*}{29.14}&\multirow{4}{*}{29.10}& 
    $\mc{E}^{DFT}_{1^2}= 2.91e^{-06}$  & $\mc{E}^{M_1}_{1^2}=2.54e^{-03}$ & $\mc{E}^{M}_{1^2}=2.90e^{-09}$
    \\ 
& &  & & $\mc{E}^{DFT}_{2}= 2.15e^{-12}$  & $\mc{E}^{M_1}_{2}=1.64e^{-09}$ & $\mc{E}^{M}_{2}=7.34e^{-10}$
    \\  
          &  & & & $\mc{E}^{DFT}_{5}= 1.52e^{-10}$  & $\mc{E}^{M_1}_{5}=6.02e^{-09}$ & $\mc{E}^{M}_{5}=7.39e^{-09}$
        \\ 
    \hline
    \end{tabular}
       \label{tab:error-gearmat}  
    \end{center}
\end{table}
\begin{figure}[H]
\centering
\subfigure[Frontal slices are $gearmat$ matrices of order $n$]{\includegraphics[scale=0.43]{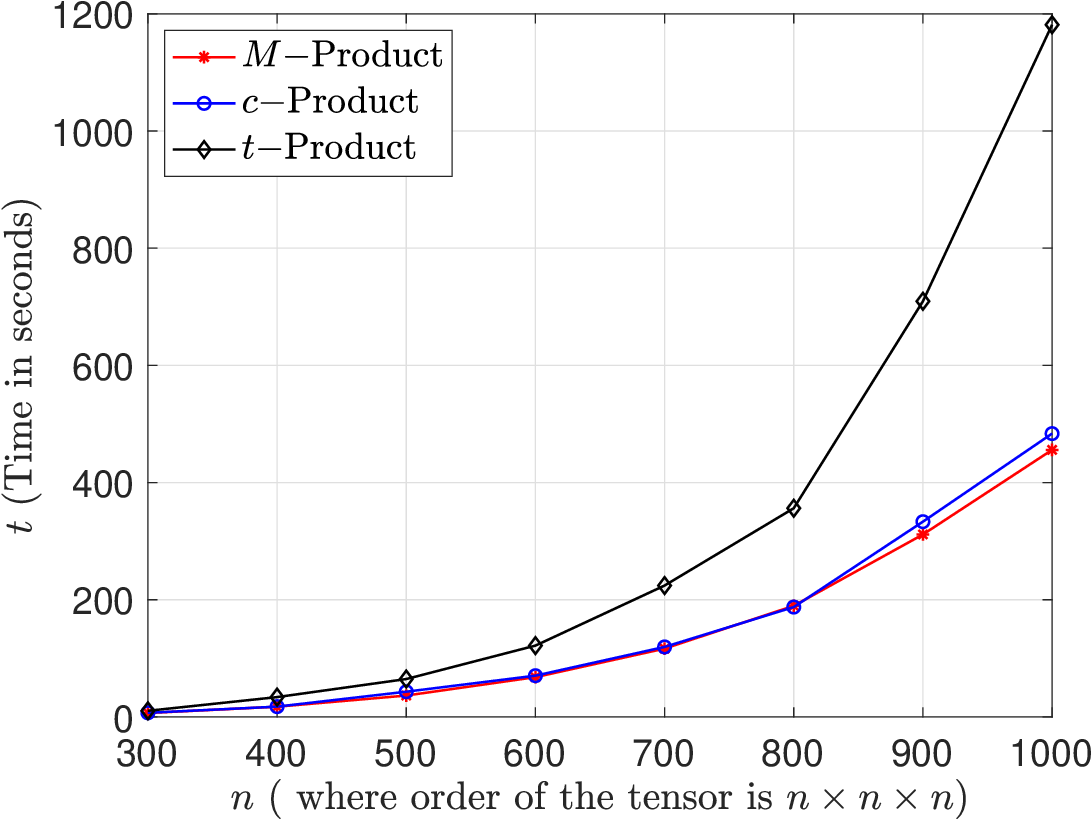}}
\subfigure[Randomly generated tensors with $ind(\mc{A})=1$]{\includegraphics[scale=0.43]{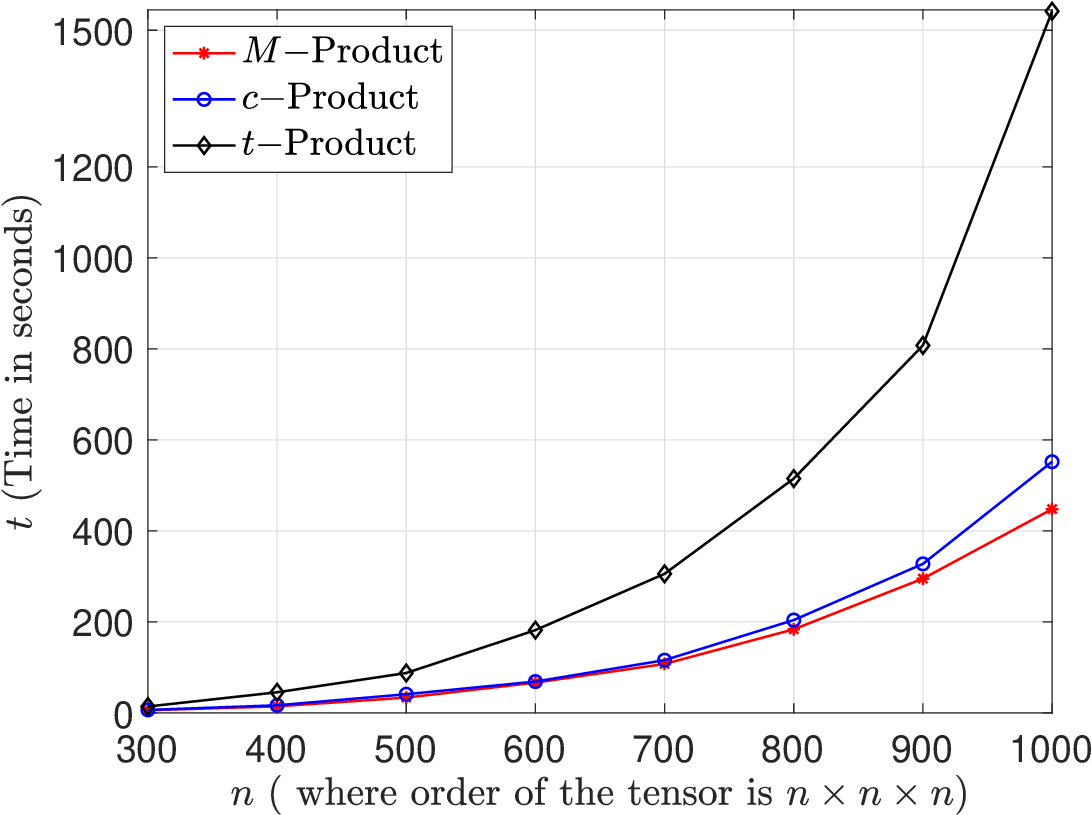}}
\caption{Comparison analysis of mean CPU time and order of tensor for computing the Drazin inverse.}
\label{fig:comp-MDR}
\end{figure}
It was observed that, in both cases, the computational efficiency is better in the $ M$ product. So, in future research, it would be interesting to see what class of invertible matrices has a low computational cost. 

\subsection{Examples based on hyperpower iterative method}
In this subsection, we have worked out a few numerical examples to study the computational efficiency of the proposed algorithm based on $M$-product. In view of corollary \ref{cor-4-2}, we compute the Moore-Penrose and Drazin inverse in the following examples.
\begin{example}\rm
Let $M=\begin{pmatrix}
1 &0 &-1 &0\\
0& 1 &0 &0\\
0& 0 & 0& 1\\
0& 1& 1& 0   
\end{pmatrix}$.   Consider a tensor  $\mc{A}\in \mb{R}^{2\times 2\times 4}$ with
    \[\mc{A}(:,:,1)=\begin{pmatrix}
        -1 & -1\\
        0 & 1
    \end{pmatrix},~\mc{A}(:,:,2)=\begin{pmatrix}
        1 & 0\\
        0 & -1
    \end{pmatrix},~\mc{A}(:,:,3)=\begin{pmatrix}
        1 & -1\\
        -1 & 0
    \end{pmatrix},~\mbox{ and }\mc{A}(:,:,4)=\begin{pmatrix}
        -1 & -1\\
        1 & 1
    \end{pmatrix}.\]
  By applying the Algorithm \ref{alg:MHPI}, we obtain  $\mc{X}=\mc{A}^{\dagger}$, where  
 \[\mc{X}(:,:,1)=\begin{pmatrix}
-7/6    &       -1/3     \\
	       1/6     &       4/3   
     \end{pmatrix},~\mc{X}(:,:,2)=\begin{pmatrix}
 1&   0\\
	  0& -1
  \end{pmatrix},~\mc{X}(:,:,3)=\begin{pmatrix}
   -2/3     &      -1/3  \\   
	      -1/3       &     1/3     
  \end{pmatrix},~\mc{X}(:,:,4)=\begin{pmatrix}
 -1/4     &       1/4     \\
	      -1/4     &       1/4     
  \end{pmatrix}.\] 
The errors, number of iterations (ITR), and mean CPU time in seconds (MT) associated with the approximated Moore-Penrose inverse are provided below.

\begin{table}[H]
    \centering
    \begin{tabular}{ccccc}
    \hline
     $\mc{Z}_0$    &  $\epsilon$ & MT & ITR & Errors \\
       \hline
          $\frac{\mc{A}^T}{\|\mc{A}\|_{F}^2}$  & $10^{-15}$&$4.375e^{-04}$ &2&$\mc{E}_1^{M}= 4.57e^{-16},~\mc{E}_2^{M} = 5.55e^{-17},~\mc{E}_3^{M} = 7.85e^{-17},~\mc{E}_4^{M}= 1.57e^{-16}$ \\
    \hline
    \end{tabular}
    %\caption{Caption}
    \label{tab-tenmpi19-exa}
\end{table}
\end{example}

\begin{example}\rm
   Let $M=\begin{pmatrix}
   1 &-3 &0\\
   1& -3 &1\\
   1 &-1 &-1
  \end{pmatrix}$. Consider $\mc{A}\in \mb{R}^{3\times 3\times 3}$ with
  \[\mc{A}(:,:,1)=\begin{pmatrix}
     1 &-1 &-1\\
     1 & 1 &1\\
     -1 & 1 & 1
  \end{pmatrix},~\mc{A}(:,:,2)=\begin{pmatrix}
     1 &0 &0\\
     0 & 0 &0\\
    0 & 0 &0
  \end{pmatrix},~\mc{A}(:,:,3)=\begin{pmatrix}
     1 &1 &1\\
     -1 & -1 &-1\\
     0 & 0 & 0
  \end{pmatrix}.\]
Clearly $\ind(\mc{A})=1$ since $\rank(\mat(\mc{A}))=6= \rank(\mat(\mc{A}^2))$. By applying the Algorithm \ref{alg:MHPI}, we obtain  $\mc{Z}=\mc{A}^{D}$, where  
 {\small{
 \[\mc{Z}(:,:,1)=\begin{pmatrix}
  -5      &       -5/2  &         -5/2   \\  
	      29/2     &      17/2    &       17/2 \\  
	     -29/2     &     -13/2    &      -13/2    
 \end{pmatrix},~
     \mc{Z}(:,:,2)=\begin{pmatrix}
   -{3}/{2}      &     -3/4   &        -3/4  \\   
	      19/4     &      11/4   &        11/4  \\   
	     -19/4       &    -9/4&           -9/4    
  \end{pmatrix},~\mc{Z}(:,:,3)=\begin{pmatrix}
  -1/2   &         1/4      &      1/4\\     
	      -1/4      &     -1/4      &     -1/4     \\
	      -3/4        &    3/4    &        3/4     
  \end{pmatrix}.
  \]} 
  }
  The errors, number of iterations (ITR), and mean CPU time in seconds (MT) associated with the approximated Drazin inverse are provided below.
\begin{table}[H]
    \centering
    \begin{tabular}{ccccc}
    \hline
     $\mc{Z}_0$    &  $\epsilon$ & MT & ITR & Errors \\
       \hline

   $0.1624\mc{A}$ &$10^{-10}$ & $3.38e^{-05}$&3&$\mc{E}_1^{M}=3.18e^{-13},~\mc{E}_2^{M} = 2.62e^{-11},~\mc{E}_5^{M} = 7.87e^{-13}
$  \\
    \hline
    \end{tabular}
    %\caption{Caption}
    \label{tab-tengr19-exa}
\end{table}
\end{example}
In Table \ref{tab:error-m19-m9}, we compare mean CPU time and errors associated with $M$-HPI9 and $M$-HPI19. Further comparison with two different choices of $M$ is illustrated in Figure \ref{fig:Fig4}. The results indicate that the computational time of $M$-HPI19 is faster than $M$-HPI9. Thus, in our next discussion, we only restrict our comparison analysis to $M$-HPI19.

\begin{table}[H]
    \begin{center}
          \caption{Error and mean CPU time for computing the Moore-Penrose inverse by $M$-HPI19 and $M$-HPI9.}
         \vspace*{0.2cm}
         \renewcommand{\arraystretch}{1.2}
    \begin{tabular}{ccccc}
    \hline
        Size of $\mc{A}$  & MT$^{M-\mbox{\scriptsize HPI9
        }}$  & MT$^{M-\mbox{\scriptsize HPI19}}$ & Error$^{M-\mbox{\scriptsize HPI9}}$ &Error$^{M-\mbox{\scriptsize HPI19}}$\\ 
           \hline
    \multirow{4}{*}{$150\times 150\times 150$} &  \multirow{4}{*}{1.11} &\multirow{4}{*}{1.06}& 
    $\mc{E}^{M-\mbox{\scriptsize HPI9}}_1=2.70e^{-08}$ & $\mc{E}^{M-\mbox{\scriptsize HPI19}}_1=1.31e^{-12}$
    \\ 
& &   & $\mc{E}^{M-\mbox{\scriptsize HPI9}}_{2}=4.03e^{-09}$ & $\mc{E}^{M-\mbox{\scriptsize HPI19}}_{2}=3.53e^{-07}$
    \\  
   
    & &   & $\mc{E}^{M-\mbox{\scriptsize HPI9}}_{3}=4.57e^{-12}$ & $\mc{E}^{M-\mbox{\scriptsize HPI19}}_{3}=1.16e^{-10}$
    \\  
    & &   & $\mc{E}^{M-\mbox{\scriptsize HPI9}}_{4}=1.25e^{-12}$ & $\mc{E}^{M-\mbox{\scriptsize HPI19}}_{4}=8.09e^{-11}$
    \\  
    \hline
       \multirow{4}{*}{$300\times 300\times 300$} &  \multirow{4}{*}{6.62} &\multirow{4}{*}{6.35}& 
    $\mc{E}^{M-\mbox{\scriptsize HPI9}}_1=1.78e^{-12}$ & $\mc{E}^{M-\mbox{\scriptsize HPI19}}_1=3.85e^{-12}$
    \\ 
& &   & $\mc{E}^{M-\mbox{\scriptsize HPI9}}_{2}=1.12e^{-09}$ & $\mc{E}^{M-\mbox{\scriptsize HPI19}}_{2}=2.47e^{-09}$
    \\  
   
    & &   & $\mc{E}^{M-\mbox{\scriptsize HPI9}}_{3}=5.95e^{-12}$ & $\mc{E}^{M-\mbox{\scriptsize HPI19}}_{3}=2.32e^{-11}$
    \\  
    & &   & $\mc{E}^{M-\mbox{\scriptsize HPI9}}_{4}=1.10e^{-12}$ & $\mc{E}^{M-\mbox{\scriptsize HPI19}}_{4}=5.51e^{-12}$
    \\  
    \hline
        \multirow{4}{*}{$450\times 450\times 450$} &  \multirow{4}{*}{26.85} &\multirow{4}{*}{25.07}& 
    $\mc{E}^{M-\mbox{\scriptsize HPI9}}_1=2.27e^{-12}$ & $\mc{E}^{M-\mbox{\scriptsize HPI19}}_1=1.78e^{-12}$
    \\ 
& &   & $\mc{E}^{M-\mbox{\scriptsize HPI9}}_{2}=3.08e^{-09}$ & $\mc{E}^{M-\mbox{\scriptsize HPI19}}_{2}=1.12e^{-09}$
    \\  
   
    & &   & $\mc{E}^{M-\mbox{\scriptsize HPI9}}_{3}=1.00e^{-11}$ & $\mc{E}^{M-\mbox{\scriptsize HPI19}}_{3}=1.59e^{-11}$
    \\  
    & &   & $\mc{E}^{M-\mbox{\scriptsize HPI9}}_{4}=2.06e^{-12}$ & $\mc{E}^{M-\mbox{\scriptsize HPI19}}_{4}=2.91e^{-12}$
    \\  
    \hline
    \end{tabular}
       \label{tab:error-m19-m9}  
    \end{center}
\end{table}

  \begin{figure}[H]
\centering
\subfigure[Frontal slices and $M$ are generated randomly]{\includegraphics[scale=0.43]{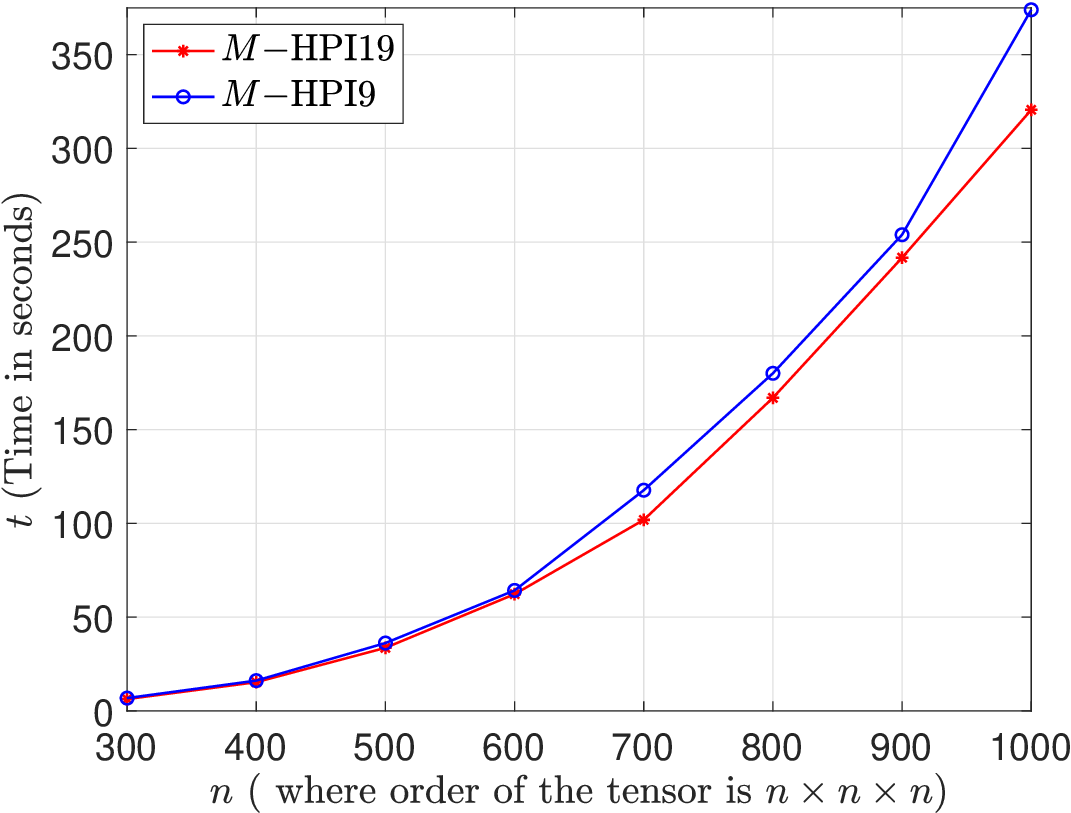}}
\subfigure[\scriptsize{Randomly generated tensors with $M=dftmtx(n)$}]{\includegraphics[scale=0.43]{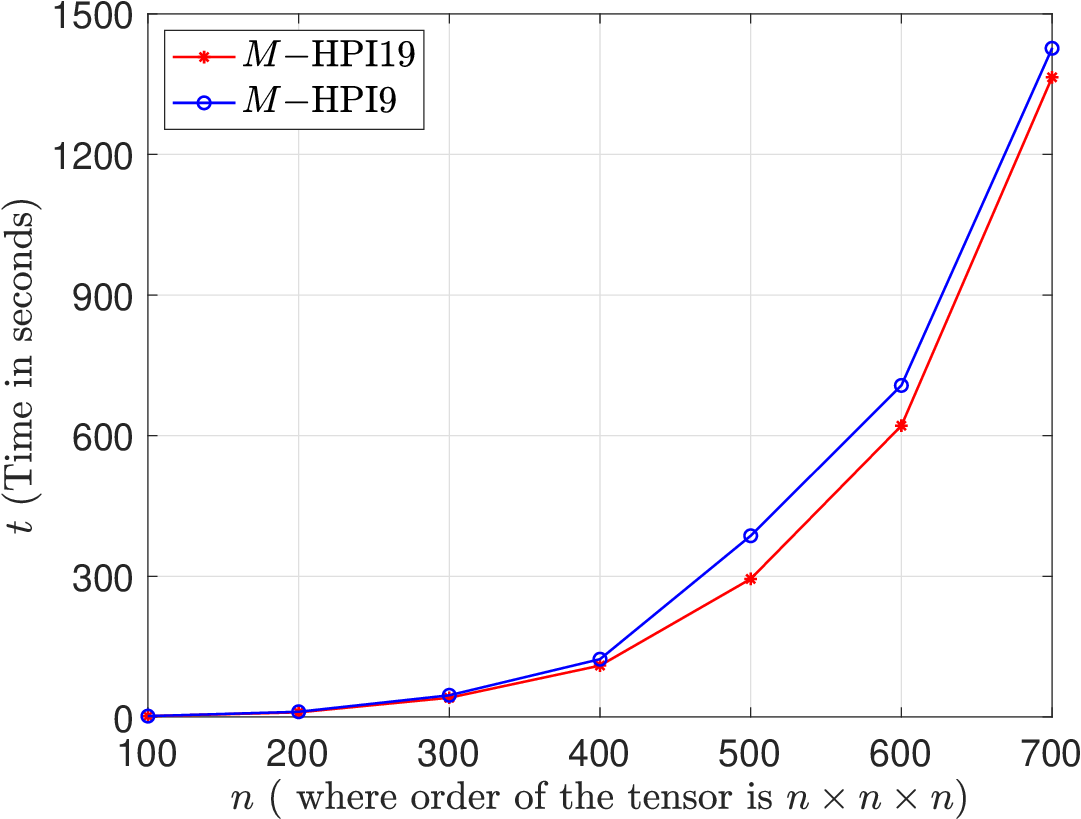}}
\caption{Comparison analysis of mean CPU time of $M$-HPI9 and $M$-HPI19 for different choices of $M$ for computing the Moore-Penrose inverse of the singular tensor $\mc{A}$ with taking $\mc{Z}_0=\frac{\mc{A}^{T}}{\|\mc{A}\|^{2}_F}$ and tolerance $\epsilon=10^{-08}$.}
\label{fig:Fig4}
\end{figure}

\begin{table}[H]
 \begin{center}
    \caption{Error and mean CPU time for computing the Moore-Penrose inverse of tensors and matrices by $M$-HPI19 and  HPI19\cite{ma2022} with $\mc{Z}_0=\frac{\mc{A}^{T}}{\|\mc{A}\|^{2}_F}$,  $Z_0=\frac{A^{T}}{\|A\|^{2}_F}$, tolerance $\epsilon=10^{-8}$ and $M$ is generated randomly.}
         \vspace*{0.2cm}
         \renewcommand{\arraystretch}{1.2}
         \begin{tabular}{cccccc}

        Size of $\mc{A}$ & Size of $A$ & MT$^{\tiny{\mbox{HPI19}}}$  & MT$^{\tiny{M-\mbox{HPI19}}}$ & Error$^{\mbox{\scriptsize HPI9}}$ &Error$^{M-\mbox{\scriptsize HPI19}}$\\ 
           \hline
    \multirow{4}{*}{\small $300\times 300\times 400$} &   \multirow{4}{*}{\small $6000\times 6000$} &\multirow{4}{*}{26.52} &\multirow{4}{*}{7.21}& 
    $\mc{E}^{\mbox{\scriptsize HPI9}}_1=3.55e^{-13}$ & $\mc{E}^{M-\mbox{\scriptsize HPI19}}_1=4.58e^{-13}$
    \\ 
& &  & & $\mc{E}^{\mbox{\scriptsize HPI9}}_{2}=5.02e^{-17}$ & $\mc{E}^{M-\mbox{\scriptsize HPI19}}_{2}=7.39e^{-10}$
    \\  
   
    & &  & & $\mc{E}^{\mbox{\scriptsize HPI9}}_{3}=2.61e^{-14}$ & $\mc{E}^{M-\mbox{\scriptsize HPI19}}_{3}=9.61e^{-11}$
    \\  
    & &  & & $\mc{E}^{\mbox{\scriptsize HPI9}}_{4}=1.19e^{-15}$ & $\mc{E}^{M-\mbox{\scriptsize HPI19}}_{4}=4.24e^{-11}$
\\
 \hline
    \multirow{4}{*}{\small $400\times 400\times 400$} &   \multirow{4}{*}{\small $8000\times 8000$} &\multirow{4}{*}{63.12} &\multirow{4}{*}{13.16}& 
    $\mc{E}^{\mbox{\scriptsize HPI9}}_1=1.13e^{-12}$ & $\mc{E}^{M-\mbox{\scriptsize HPI19}}_1=4.12e^{-12}$
    \\ 
& &  & & $\mc{E}^{\mbox{\scriptsize HPI9}}_{2}=3.27e^{-17}$ & $\mc{E}^{M-\mbox{\scriptsize HPI19}}_{2}=1.13e^{-10}$
    \\  
   
    & &  & & $\mc{E}^{\mbox{\scriptsize HPI9}}_{3}=2.68e^{-14}$ & $\mc{E}^{M-\mbox{\scriptsize HPI19}}_{3}=3.67e^{-11}$
    \\  
    & &  & & $\mc{E}^{\mbox{\scriptsize HPI9}}_{4}=3.25e^{-12}$ & $\mc{E}^{M-\mbox{\scriptsize HPI19}}_{4}=4.90e^{-12}$
\\
\hline
\multirow{4}{*}{\small $500\times 500\times 400$} &   \multirow{4}{*}{\small $10000\times 10000$} &\multirow{4}{*}{134.49} &\multirow{4}{*}{24.42}& 
    $\mc{E}^{\mbox{\scriptsize HPI9}}_1=1.21e^{-12}$ & $\mc{E}^{M-\mbox{\scriptsize HPI19}}_1=5.25e^{-12}$
    \\ 
& &  & & $\mc{E}^{\mbox{\scriptsize HPI9}}_{2}=2.33e^{-17}$ & $\mc{E}^{M-\mbox{\scriptsize HPI19}}_{2}=1.20e^{-10}$
    \\  
   
    & &  & & $\mc{E}^{\mbox{\scriptsize HPI9}}_{3}=2.68e^{-14}$ & $\mc{E}^{M-\mbox{\scriptsize HPI19}}_{3}=1.75e^{-12}$
    \\  
    & &  & & $\mc{E}^{\mbox{\scriptsize HPI9}}_{4}=1.05e^{-15}$ & $\mc{E}^{M-\mbox{\scriptsize HPI19}}_{4}=3.26e^{-12}$
\\
\hline
     \end{tabular}
       \label{tab:error-tenm19-mat9}  
        \end{center}
\end{table}
Table \ref{tab:error-tenm19-mat9} and Figure \ref{fig:Fig5} show that the tensor structure computation under $M$-product is much faster than the matrix structured computation.

\begin{figure}[H]
\centering
\includegraphics[scale=0.53]{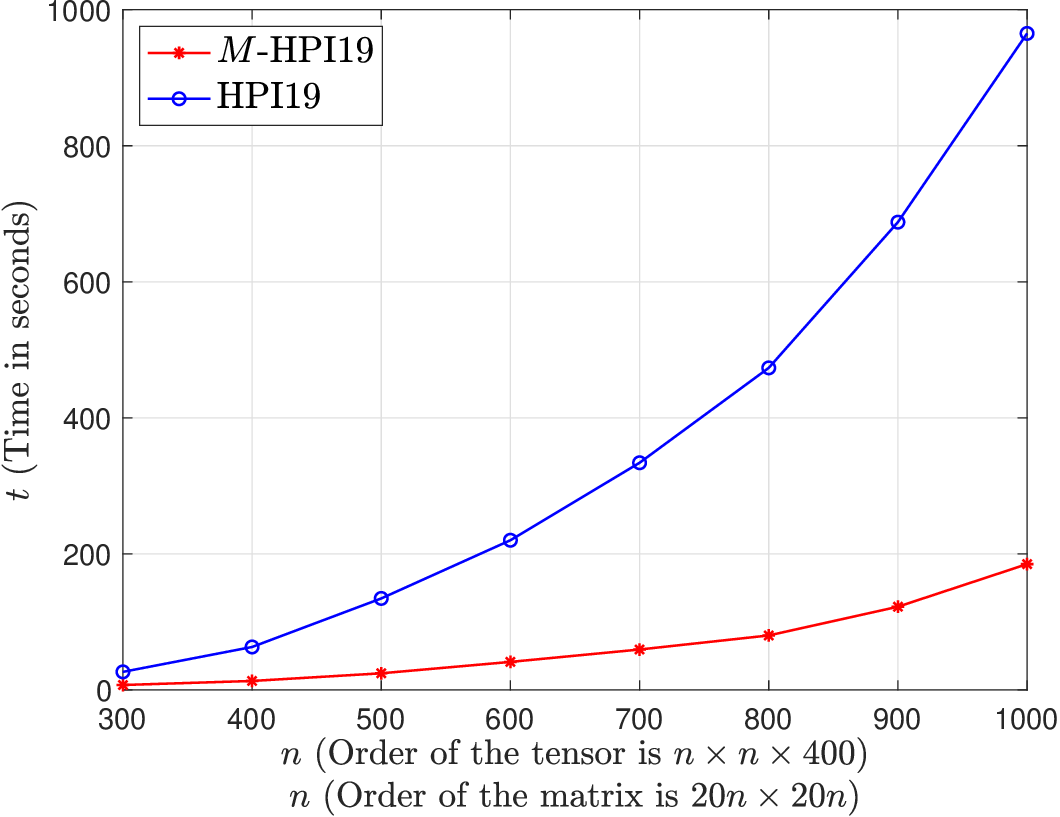}
\caption{Comparison analysis of mean CPU time of $M$-HPI9 and HPI19 \cite{ma2022}  for computing the Moore-Penrose inverse of the singular tensor $\mc{A}$ with taking $\mc{Z}_0=\frac{\mc{A}^{T}}{\|\mc{A}\|^{2}_F}$, tolerance $\epsilon=10^{-08}$ and $M$ is randomly generated.}
\label{fig:Fig5}
\end{figure}
In the Table \ref{tab:error-M-HPI19} and Figure \ref{fig:Fig6}, we have illustrated the computational time for evaluating the Moore-Penrose inverse for different choices of $M$.
  \begin{table}[H]
    \begin{center}
       \caption{ Comparison analysis of Error and mean CPU time for computing the Moore-Penrose inverse by $M$-HPI19, for different choices of $M$  with taking $\mc{Z}_0=\frac{\mc{A}^{T}}{\|\mc{A}\|^{2}_F}$ and tolerance $\epsilon=10^{-08}$.}
       \vspace{0.2cm}
          \renewcommand{\arraystretch}{1.2}
    \begin{tabular}{ccccccc}
    \hline
        Size of $\mc{A}$  & MT$^{DFT}$  & MT$^{M_1}$ & MT$^M$ & Error$^{DFT}$ & Error$^{M_1}$ &Error$^M$\\ 
           \hline
    \multirow{4}{*}{$150\times 150\times 150$} &  \multirow{4}{*}{5.142} &\multirow{4}{*}{2.83}&\multirow{4}{*}{1.31}& 
    $\mc{E}^{DFT}_{1}= 2.25e^{-09}$  & $\mc{E}^{M_1}_{1}=2.41e^{-07}$ & $\mc{E}^{M}_{1}=1.31e^{-12}$
    \\ 
& &  & & $\mc{E}^{DFT}_{2}= 8.82e^{-08}$  & $\mc{E}^{M_1}_{2}=1.130e^{-07}$ & $\mc{E}^{M}_{2}=2.46e^{-10}$
    \\  
          &  & & & $\mc{E}^{DFT}_{3}= 4.52e^{-12}$  & $\mc{E}^{M_1}_{3}=1.01e^{-10}$ & $\mc{E}^{M}_{3}=4.05e^{-12}$
    \\ 
     &  & & &$\mc{E}^{DFT}_{4}= 1.69e^{-09}$  & $\mc{E}^{M_1}_{4}=3.77e^{-08}$ & $\mc{E}^{M}_{4}=1.10e^{-12}$
    \\ 
    \hline
    \multirow{4}{*}{$300\times 300\times 300$} &  \multirow{4}{*}{36.05} &\multirow{4}{*}{26.59}&\multirow{4}{*}{6.84}& 
    $\mc{E}^{DFT}_{1}= 7.85e^{-09}$  & $\mc{E}^{M_1}_{1}=8.34e^{-07}$ & $\mc{E}^{M}_{1}=3.85e^{-12}$
    \\ 
& &  & & $\mc{E}^{DFT}_{2}= 3.64e^{-08}$  & $\mc{E}^{M_1}_{2}=9.74e^{-08}$ & $\mc{E}^{M}_{2}=3.13e^{-09}$
    \\  
          &  & & & $\mc{E}^{DFT}_{3}= 7.06e^{-12}$  & $\mc{E}^{M_1}_{3}=2.33e^{-10}$ & $\mc{E}^{M}_{3}=2.33e^{-11}$
    \\ 
     &  & & & $\mc{E}^{DFT}_{4}= 4.45e^{-09}$  & $\mc{E}^{M_1}_{4}=9.89e^{-07}$ & $\mc{E}^{M}_{4}=3.166e^{-12}$
    \\ 
    \hline
        \multirow{4}{*}{$450\times 450\times 450$} &  \multirow{4}{*}{168.66} &\multirow{4}{*}{77.70}&\multirow{4}{*}{26.95}& 
    $\mc{E}^{DFT}_{1}= 1.21e^{-08}$  & $\mc{E}^{M_1}_{1}=6.97e^{-07}$ & $\mc{E}^{M}_{1}=5.05e^{-12}$
    \\ 
& &  & & $\mc{E}^{DFT}_{2}= 1.13e^{-07}$  & $\mc{E}^{M_1}_{2}=1.44e^{-07}$ & $\mc{E}^{M}_{2}=1.83e^{-06}$
    \\  
          &  & & & $\mc{E}^{DFT}_{3}= 8.52e^{-12}$  & $\mc{E}^{M_1}_{3}=7.47e^{-10}$ & $\mc{E}^{M}_{3}=4.26e^{-10}$
    \\ 
     &  & & & $\mc{E}^{DFT}_{4}= 4.07e^{-09}$  & $\mc{E}^{M_1}_{4}=4.09e^{-06}$ & $\mc{E}^{M}_{4}=6.04e^{-11}$
    \\ 
    \hline
    \end{tabular}
       \label{tab:error-M-HPI19}  
    \end{center}
\end{table}

\begin{figure}[H]
\centering
{\includegraphics[scale=0.5]{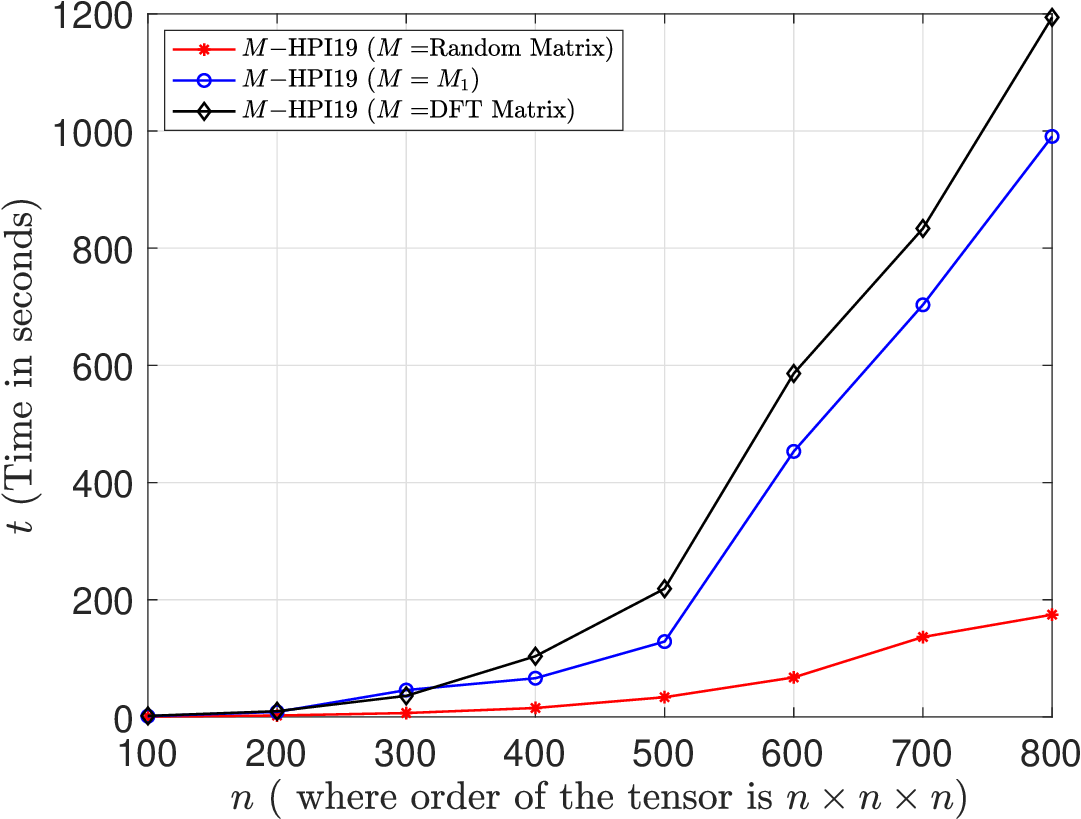}}
\caption{Comparison analysis of mean CPU time of $M$-HPI19 for different choices of $M$ for computing the Moore-Penrose inverse of singular tensor $\mc{A}$ with taking tolerance $\epsilon=10^{-07}$.}
\label{fig:Fig6}
\end{figure} 

We randomly generated tensors $\mathcal{A}$
of varying dimensions $n\times n\times n$ and considered $M$ using the Matlab in-built function $rand(n)$.  Figure \ref{fig:QR_HPI} presents a comparative analysis of computational efficiency between Algorithms \ref{AlgouterQR} and \ref{alg:MHPI}, demonstrating their relative performance across different tensor sizes.

\begin{figure}[H]
\centering
\subfigure{\includegraphics[scale=0.43]{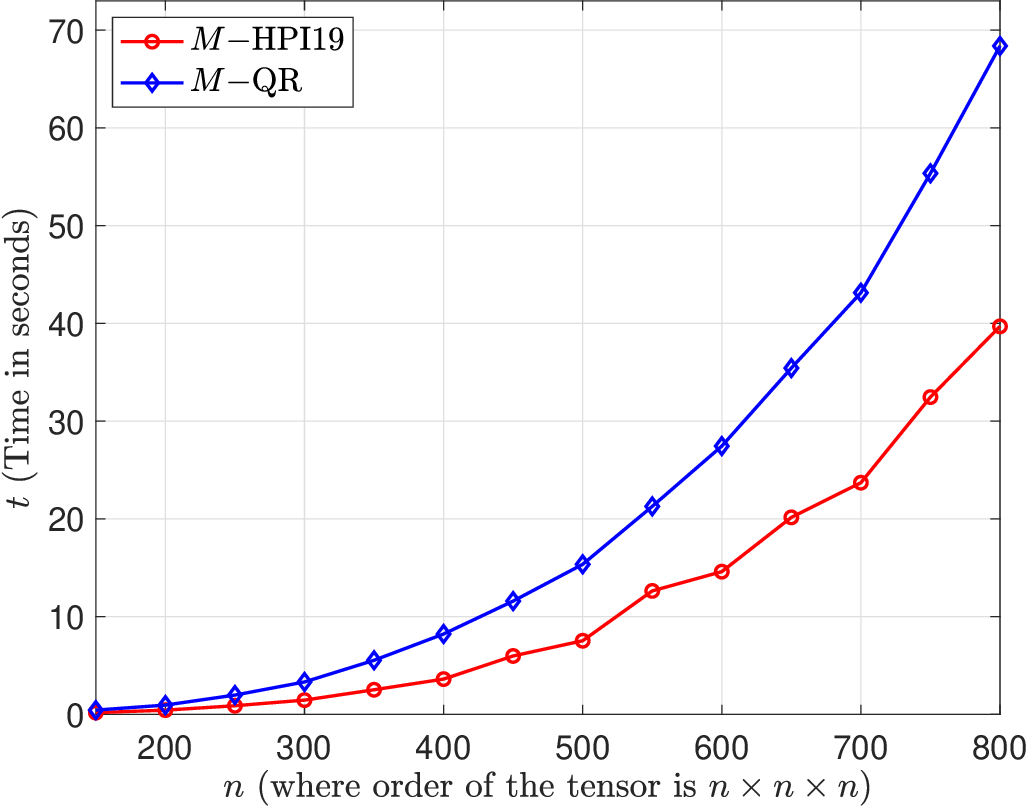}}
\subfigure{\includegraphics[scale=0.43]{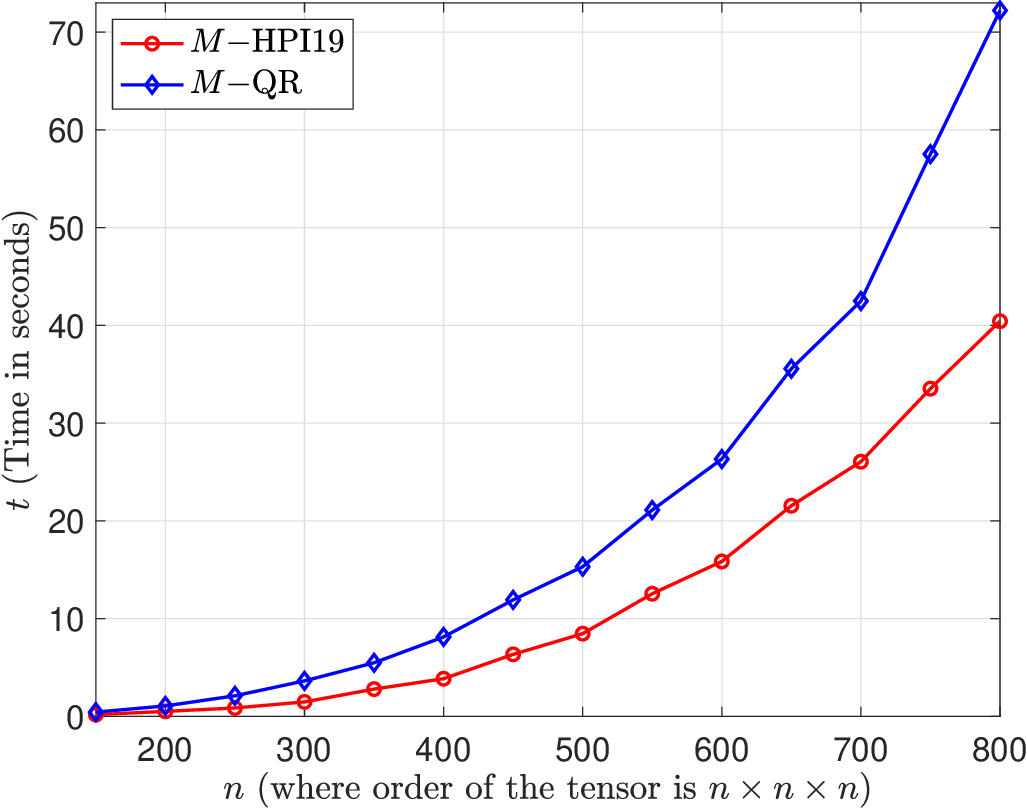}}
\caption{Comparison analysis of mean CPU time for computing the Moore-Penrose inverse by $M$-QR Algorithm \ref{AlgouterQR} and $M$-HPI19 Algorithm \ref{alg:MHPI}.}
\label{fig:QR_HPI}
\end{figure}

\section{Application in image compression and deblurring}
\subsection{Image compression}
Matrix ${QR}$ decomposition is well known for image compression by factoring a matrix into the product of an orthogonal matrix ${Q}$ and an upper triangular matrix ${R}$. In view of this, tensor ${QR}$ decomposition will offer significant advantages in digital color image and video processing applications. Specifically, a digital color image is represented as a tensor of pixel values.  Image partitioning subdivides the original image into smaller blocks, enabling localized compression that preserves regional details while improving computational efficiency. The efficacy of the approach is critically dependent on the rank of the tensor, which determines the number of significant components retained from the decomposition. We factor the true image $\mc{A}\in\mb{F}^{m\times n\times 3}$ as $\mc{A}=\mc{Q}\m\mc{R}$, where $\mc{Q}\in\mb{F}^{m\times k\times 3}$ and $\mc{R}\in\mb{F}^{k\times n\times 3}$. Here $k$ establishes the fundamental trade-off between compression ratio and image fidelity: lower values $k$ yield higher compression ratios but introduce greater information loss, while higher $k$ values preserve image quality at the expense of storage efficiency. To evaluate image quality in compression and restoration, we used peak signal-to-noise ratio (PSNR) and structural similarity index (SSIM) metrics \cite{behera2023}. PSNR quantifies image quality by measuring the ratio between the maximum possible power of a signal and the power of corrupting noise, expressed in logarithmic decibel scale; higher PSNR values generally indicate better quality reconstructions. On the other hand, SSIM addresses this limitation by comparing local patterns of pixel intensities normalized for luminance and contrast, measuring structural information changes, and providing values between $-1$ and $1$, where $1$ indicates perfect similarity.

\begin{figure}[H]
\begin{center}
\subfigure[size: $576 \times 720\times 3$ ]{\includegraphics[height=4.4cm]{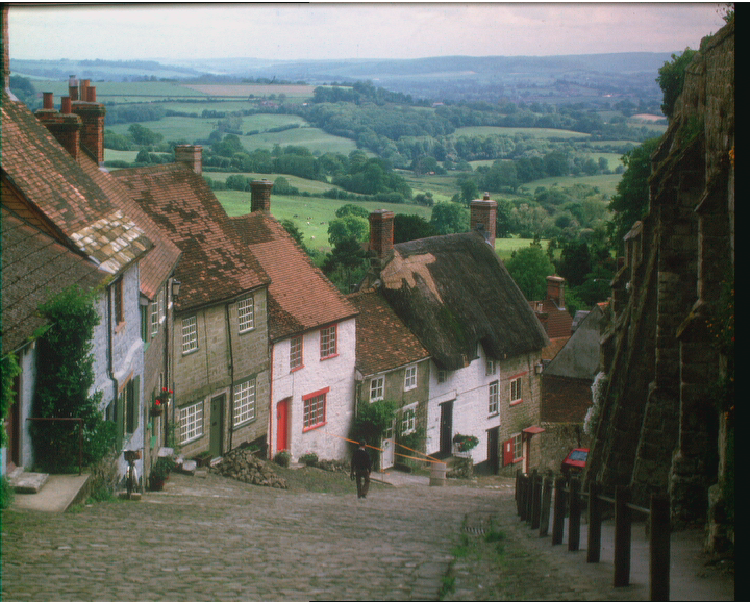}}~~~~
\subfigure[size: $512\times 512\times 3$ ]{\includegraphics[height=4.4cm]{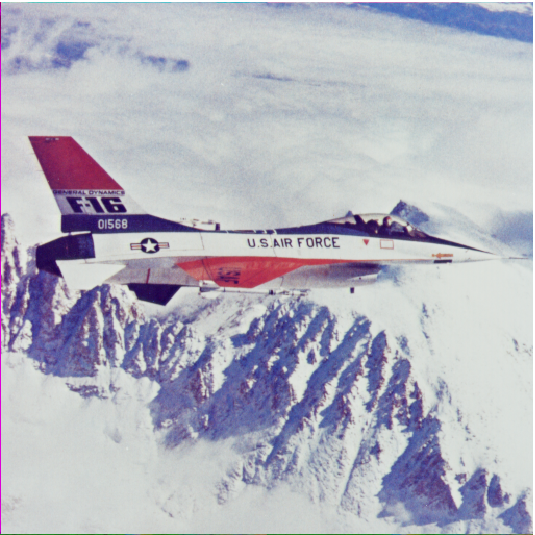}}
\caption{ Original images for compression.
}\label{or-img1}\vspace{.5cm}
\subfigure[$k=50$]{\includegraphics[height=2.3cm,width=2.7cm]{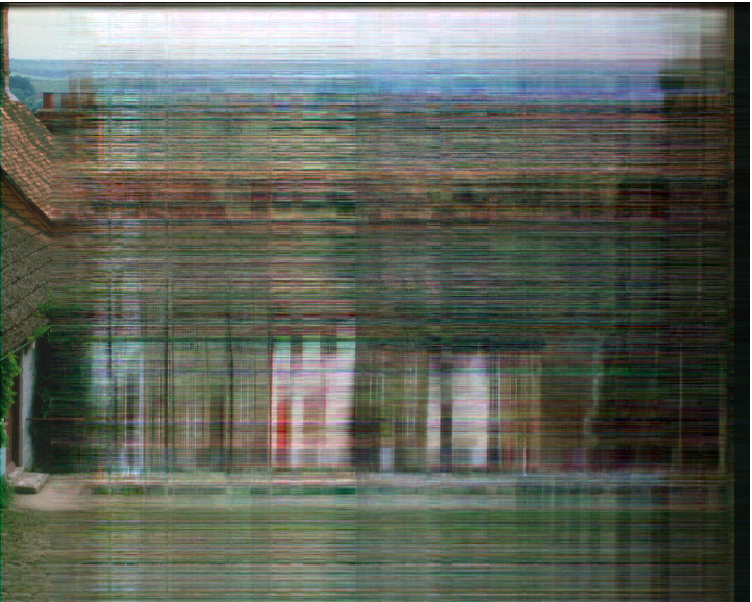}}~~
\subfigure[$k=100$]{\includegraphics[height=2.3cm,width=2.7cm]{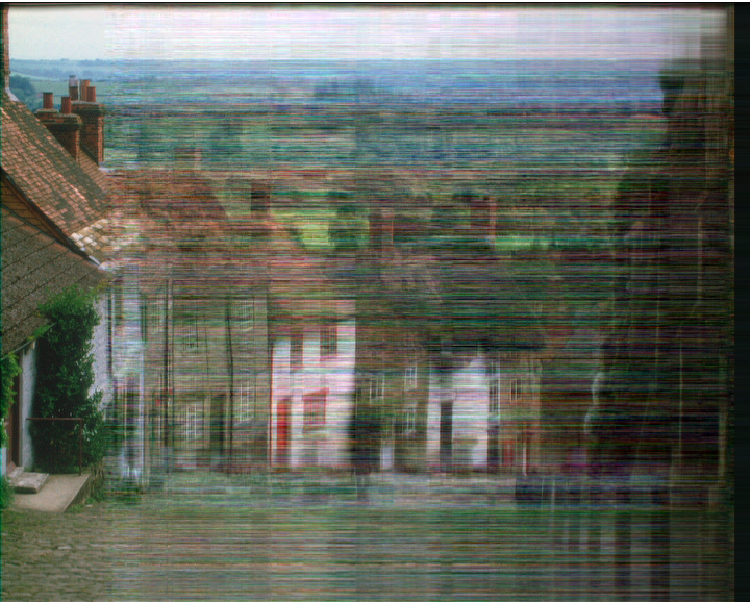}}~~
\subfigure[$k=200$]{\includegraphics[height=2.3cm,width=2.7cm]{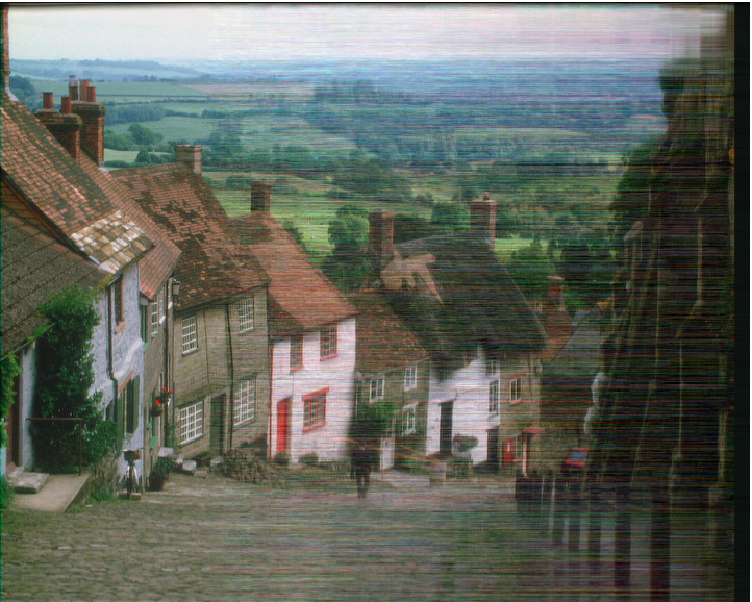}}~~
\subfigure[$k=300$]{\includegraphics[height=2.3cm,width=2.7cm]{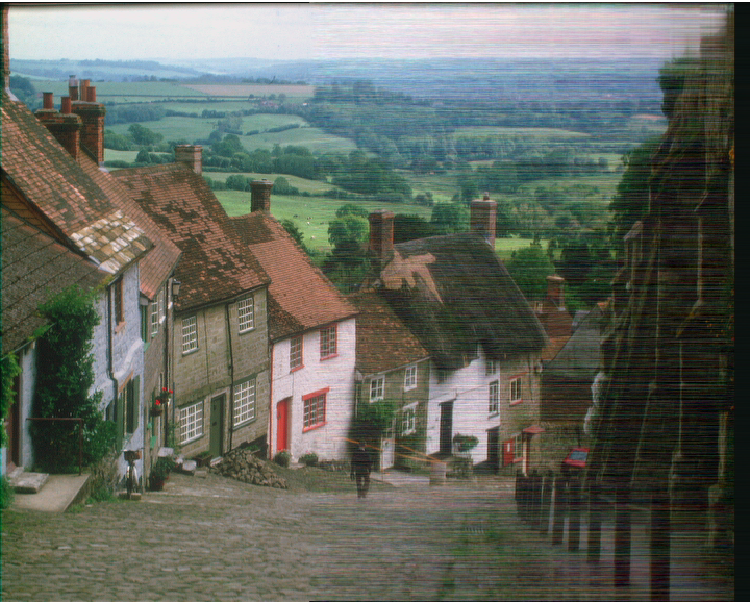}}~~
\subfigure[$k=400$]{\includegraphics[height=2.3cm,width=2.7cm]{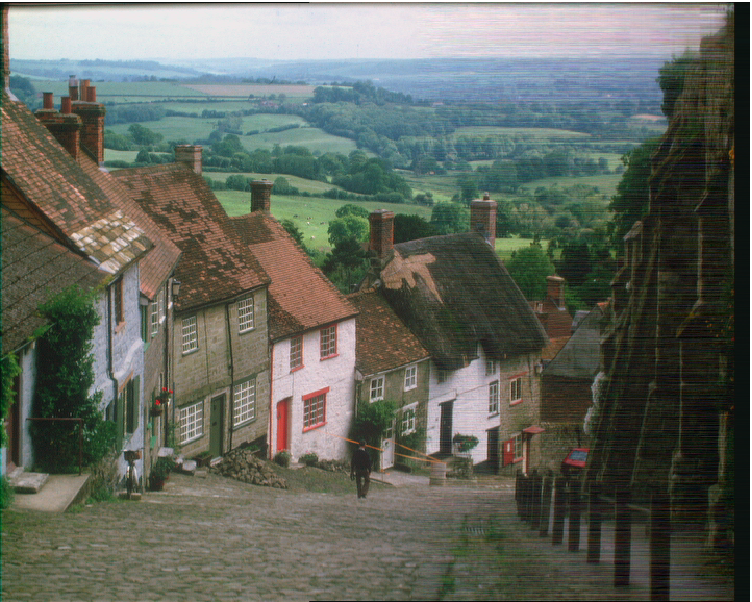}}
\caption{Compressed image of Figure \ref{or-img1} (a), for different choices of $k$.
}\label{comp-im1}\vspace{.5cm}
\subfigure[$k=50$]{\includegraphics[height=2.3cm,width=2.7cm]{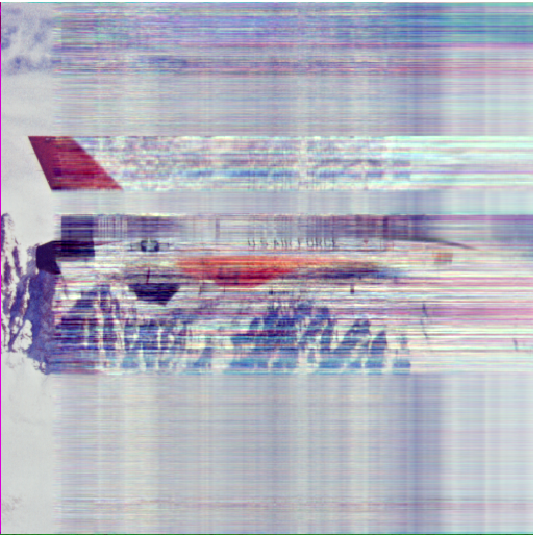}}~~
\subfigure[$k=100$]{\includegraphics[height=2.3cm,width=2.7cm]{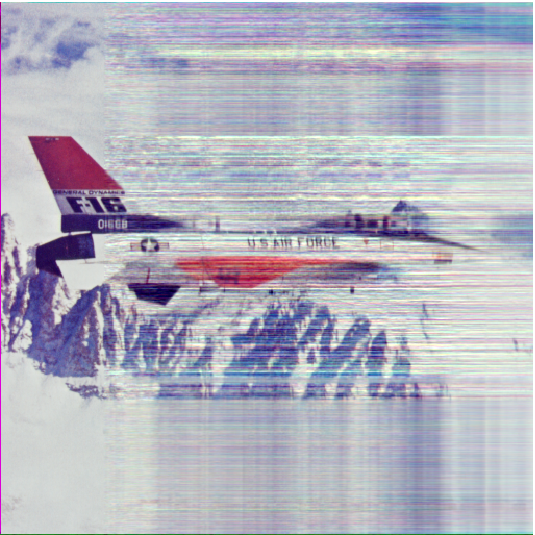}}~~
\subfigure[$k=200$]{\includegraphics[height=2.3cm,width=2.7cm]{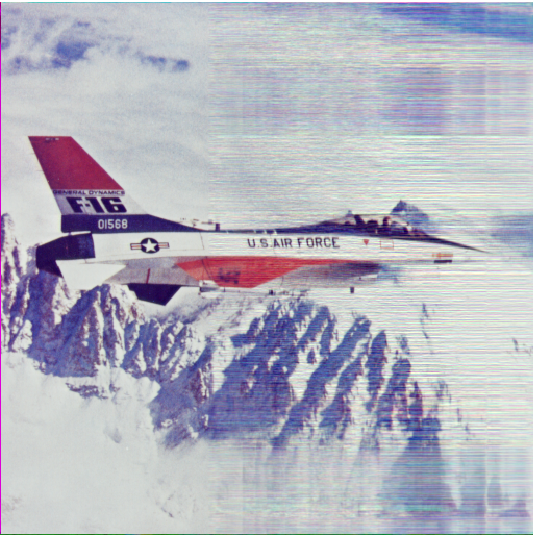}}~~
\subfigure[$k=300$]{\includegraphics[height=2.3cm,width=2.7cm]{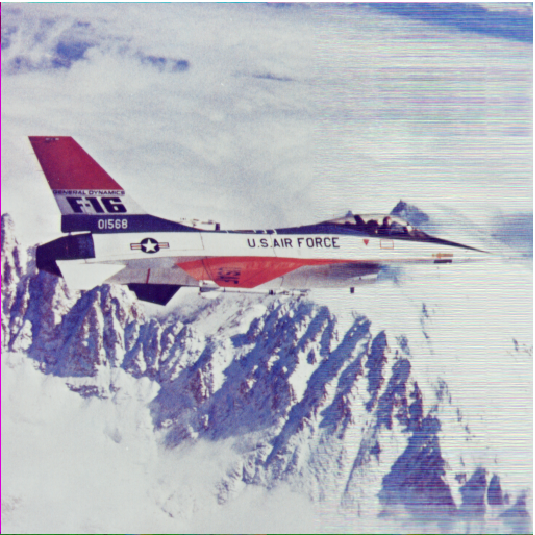}}~~
\subfigure[$k=400$]{\includegraphics[height=2.3cm,width=2.7cm]{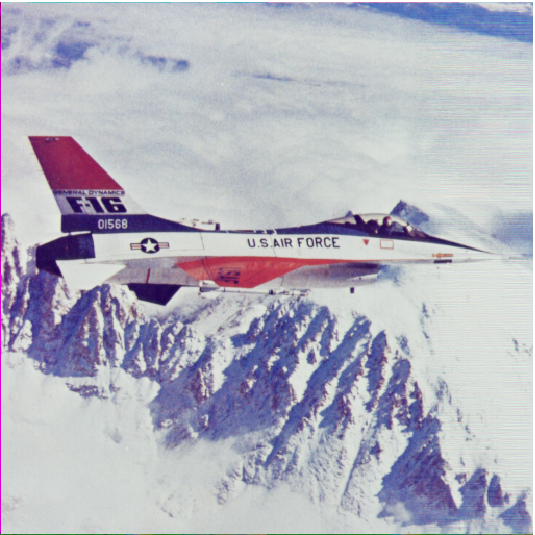}}
\caption{Compressed image of Figure \ref{or-img1} (b), for different choices of $k$.
}\label{comp-im2}\vspace{.5cm}
\subfigure{\includegraphics[scale=0.4]{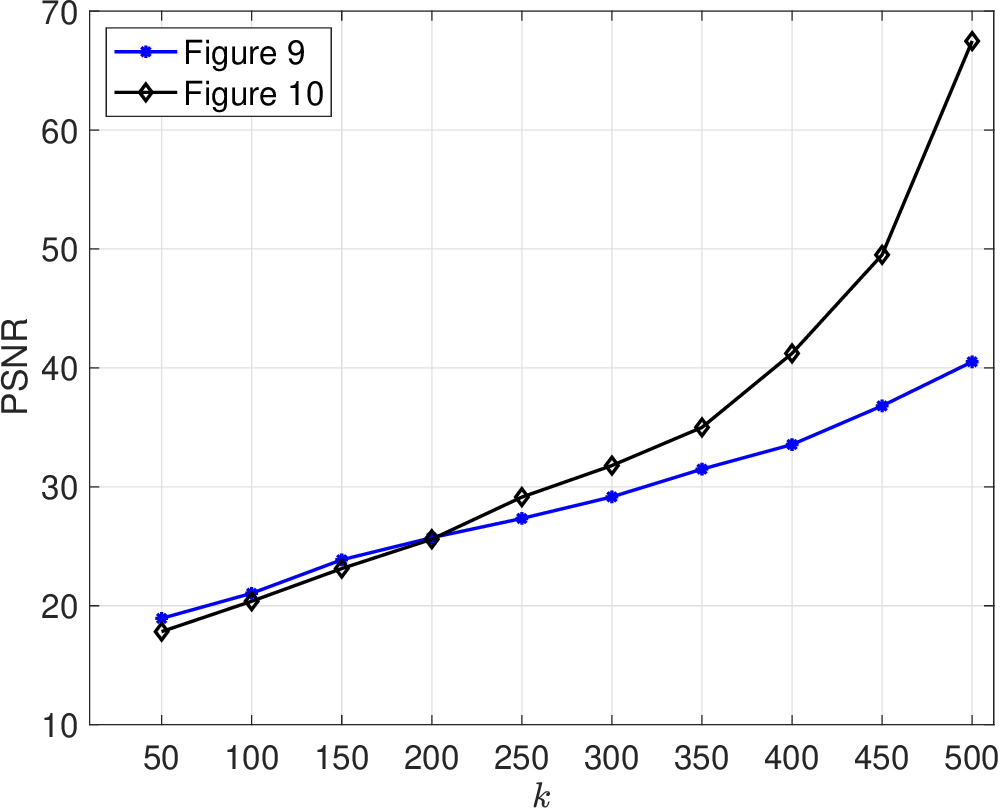}}~~~
\subfigure{\includegraphics[scale=0.4]{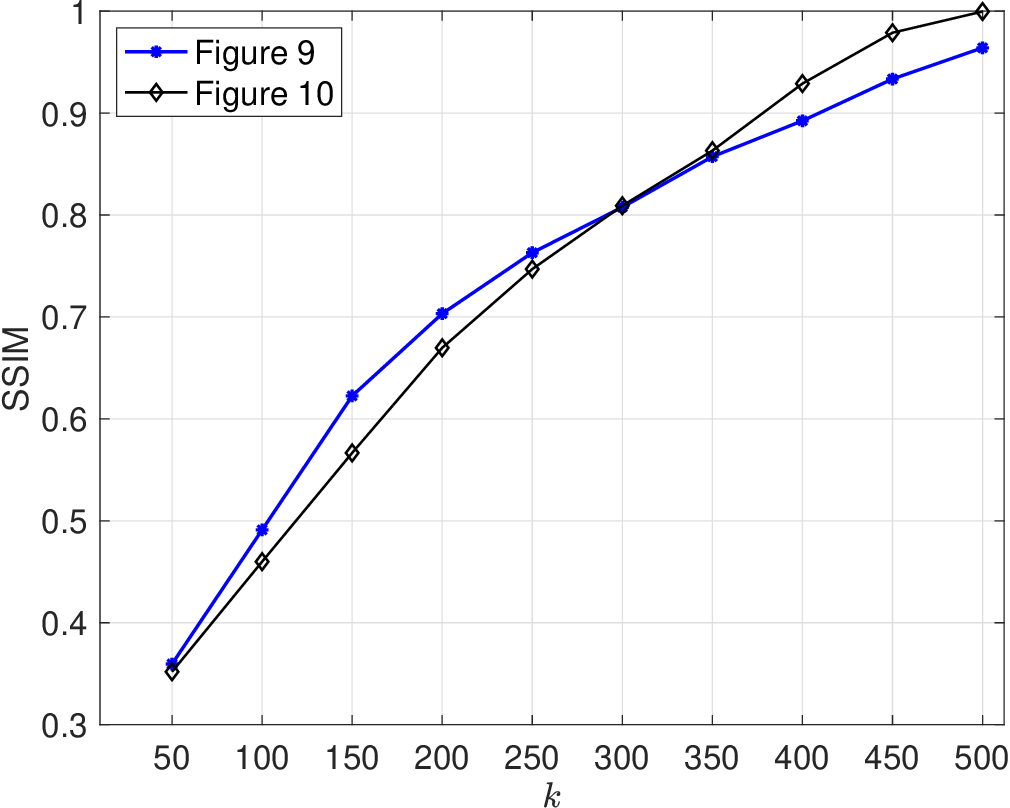}}
\caption{Quality assessment of the compressed images (Figure \ref{comp-im1} and \ref{comp-im2}), for different choices of $k$. }
\label{fig:ssim-psnr}
\end{center}
\end{figure}

In order to illustrate the accuracy and efficiency of the tensor $\mc{QR}$ decomposition. We consider two color images in Figure \ref{or-img1}; that is, Figure \ref{or-img1} (a) image size $576\times720\times 3$ and Figure \ref{or-img1} (b) image size $512\times512\times 3$. The quality assessment of the compressed image with a different value of $k$ is plotted in Figure \ref{comp-im1} for the image presented in Figure \ref{or-img1} (a). Similarly, in view of Figure \ref{or-img1} (b), the quality assessment of the compressed image with a different value of $k$ is plotted in Figure \ref{comp-im2}.

In Figure \ref{fig:ssim-psnr} (a) we have plotted the relation between $k$ and the PSNR values for the original image presented in Figures \ref{or-img1} (a) and (b). Generally, a PSNR value that exceeds $30$ dB indicates good image quality with acceptable levels of distortion.  Similarly, in Figure \ref{fig:ssim-psnr} (b), we have plotted the relation between $k$ and the SSIM values for the original image presented in Figures \ref{or-img1} (a) and (b). Here, SSIM values closer to $1$ indicate a greater structural similarity to the original images Figures \ref{or-img1} (a) and (b).

\subsection{Image deblurring}
This subsection investigates the application of the tensor-structured hyperpower iterative method (to compute the Moore-Penrose inverses) of third-order tensors to reconstruct color images from blurred images. Recent investigations by Behera {\it et al.} \cite{behera2022} and Sahoo {\it  et al.} \cite{Panda25} have made notable contributions to color image deblurring problems. The image reconstruction problem can be formulated as follows:
\begin{equation*}
    \mc{A}\m \mc{X}=\mc{B},~ \text{where~}{\mc{X}\in \mathbb{R}^{n\times n \times 3}, ~\mc{B}\in \mathbb{R}^{n\times n \times 3}},
\end{equation*}
where $\mc{A}\in \mathbb{R}^{n\times n \times 3}$ represent the blurring tensor that we are constructing using a banded Toeplitz matrix. Specifically, the frontal slices of the blurring tensor $\mc{A}^{n \times n \times 3}$ are given by
$$ \mc{A}(i,j,k) =\left\{\begin{array}{ll}
\frac{1}{\sigma \sqrt{2 \pi}} e^{-\frac{(i-j)^{2}}{2 \sigma^{2}}}, & |i-j| \leq r,\\
0, & \text { otherwise, }
\end{array}\right.$$
for $k=1,2,3$.
Here, $r$ is the bandwidth, and $\sigma$ controls the amount of smoothing, i.e., the problem is ill-posed when $\sigma$ is larger. Further, the tensor $\mc{X}$ denotes the original color image tensor, which we must find, and $\mc{B}\in \mathbb{R}^{n\times n \times 3}$ represents the observed blurred image tensor.

The efficacy of the proposed approach is evaluated on two color images, where each image of size $128\times 128\times 3$. Using the different values of $\sigma$, and fixing the value of $r=6$, we generate the blurred image $\mc{B} = \mc{A}\m\mc{X} + \mc{N}$, where $\mc{N}$ is a noise tensor, which is constructed by taking each slice as Gaussian noise with the mean $0$ and variance $10^{-4}$. The true images are provided in Figure \ref{or-img} (a) and Figure \ref{or-img} (b). The blurred and noisy images are shown in Figure \ref{bl-img1} and Figure \ref{bl-img2} with a different value of $\sigma$. With the help of the tensor-structured hyperpower iterative method, we compute the reconstructed images, which are shown in Figure \ref{recon-im1} and Figure \ref{recon-im2}. Table \ref{tab:my_label} demonstrated the quality of the reconstructed images, that is, the relation between $\sigma$, $r$ with the PSNR and SSIM values for the images presented in Figure \ref{recon-im1} and Figure \ref{recon-im2}.

\begin{figure}[H]
\begin{center}
\subfigure[]{\includegraphics[height=4.4cm]{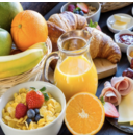}}~~~~
\subfigure[]{\includegraphics[height=4.4cm]{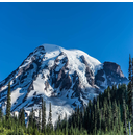}}
\caption{ Original images, where each of size $128\times 128\times 3$. 
}\label{or-img}
\subfigure[$\sigma=4$]{\includegraphics[height=4.4cm]{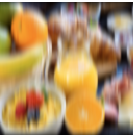}}~~~~
\subfigure[$\sigma=5$]{\includegraphics[height=4.4cm]{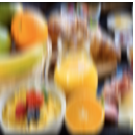}}~~~~
\subfigure[$\sigma=10$]{\includegraphics[height=4.4cm]{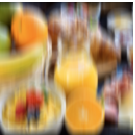}}
\caption{Blurred noisy images of Figure \ref{or-img} (a), for different choices of $\sigma$ and $r=6$.
}\label{bl-img1}
\subfigure[$\sigma=4$]{\includegraphics[height=4.4cm]{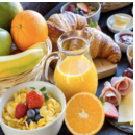}}~~~~
\subfigure[$\sigma=5$]{\includegraphics[height=4.4cm]{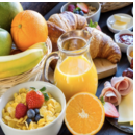}}~~~~
\subfigure[$\sigma=10$]{\includegraphics[height=4.4cm]{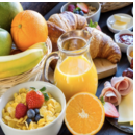}}
\caption{Reconstructed images of Figure \ref{bl-img1}, for the corresponding choices of $\sigma$ and $r=6$.
}\label{recon-im1}
\end{center}
\end{figure}

\begin{table}[H]
    \centering
    \caption{Quality assessment of the reconstructed images (Figure \ref{recon-im1} and \ref{recon-im2}), for different choices of $\sigma$ and $r$.}
     \vspace{0.2cm}
          \renewcommand{\arraystretch}{1.2}
     \begin{tabular}{ccccc|ccc}
    \hline
      $\sigma$ & $r$ &Figure    & SSIM& PSNR& Figure  & SSIM& PSNR\\
     \hline
       4 &   \multirow{3}{*}{6} & \ref{recon-im1}(a)   &   \multirow{3}{*}{0.99}& 28.22 &  \ref{recon-im2}(a)  & \multirow{3}{*}{0.99} &29.50\\
       5 &  & \ref{recon-im1}(b) &  &47.32& \ref{recon-im2}(b)&  &47.05\\
        10  &  & \ref{recon-im1}(c)& &51.45& \ref{recon-im2}(c) & &52.03\\
      \hline
    \end{tabular}
    \label{tab:my_label}
\end{table}

\begin{figure}[H]
\begin{center}
\subfigure[$\sigma=4$]{\includegraphics[height=4.4cm]{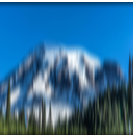}}~~~~
\subfigure[$\sigma=5$]{\includegraphics[height=4.4cm]{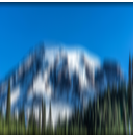}}~~~~
\subfigure[$\sigma=10$]{\includegraphics[height=4.4cm]{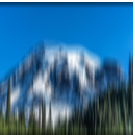}}
\caption{Blurred noisy images of Figure \ref{or-img} (b), for different choices of $\sigma$ and $r=6$.
}\label{bl-img2}
\subfigure[$\sigma=4$]{\includegraphics[height=4.4cm]{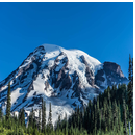}}~~~~
\subfigure[$\sigma=5$]{\includegraphics[height=4.4cm]{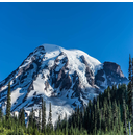}}~~~~
\subfigure[$\sigma=10$]{\includegraphics[height=4.4cm]{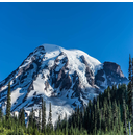}}
\caption{Reconstructed images of Figure \ref{bl-img2}, for the corresponding choices of $\sigma$ and $r=6$.
}\label{recon-im2}
\end{center}
\end{figure}

\section{Conclusion}
The outer inverse of tensors is introduced in the framework of the $M$-product, which is a family of tensor–tensor products, generalization of the $t$-product and $c$-product. In addition to this, we discussed a powerful tensor-based iterative method, the $M$-HPI method, for computing outer inverses. Further, the $M$-QR algorithm and $M$-HPI method for computing outer inverses of tensors are designed. As a consequence, we have calculated the Drazin inverse and the Moore--Penrose inverse for third-order tensors. In addition, we examine the practical applications of $M$-QR decomposition and hyperpower iterative methods for image compression and deblurring operations, demonstrating their effectiveness in real-world scenarios. Thus, the strength of the $M$-product structure is that the results obtained under the $M$-product structure are more efficient than those obtained by matrix-matrix multiplication, $t$-product, and $c$-product structure. Possible further research includes mainly the following problems:
\begin{itemize}
\item Development of adaptive tensor-based hyperpower iterative methods that automatically optimize the choice of initial approximation and $M$. 
\item Investigation of parallel computing implementations of $M$-QR decomposition and hyperpower iterative methods to compute
outer inverses of tensors to improve computational efficiency for large-scale applications.
\item Theoretical analysis of error bounds and stability conditions under various tensor perturbations.
\item Extension of the $M$-QR decomposition to handle sparse tensors with specific structural properties.
\item Integration of these methods with deep learning frameworks for hybrid computational approaches.
\end{itemize}

\section*{Data Availability Statement}
The datasets were generated or analyzed during the current study can be shared upon reasonable request to the corresponding author.

\section*{Funding}
\begin{itemize}
\item Ratikanta Behera is grateful for the support of the Anusandhan National Research Foundation (ANRF), India, under Grant No. EEQ/2022/001065. 
\item Jajati Keshari Sahoo is grateful for the support of the Anusandhan National Research Foundation (ANRF), India, under Grant No. SUR/2022/004357. 
\item Yimin Wei is grateful for the support of  the National Natural Science
Foundation of China under grant U24A2001 
and the Ministry
of Science and Technology of China under grant H20240841.
\end{itemize}

\section*{Conflict of Interest}
The authors declare that they have no potential conflict of interest.
\section*{ORCID}
Ratikanta Behera\orcidA \href{https://orcid.org/0000-0002-6237-5700}{ \hspace{2mm}\textcolor{lightblue}{ https://orcid.org/0000-0002-6237-5700}}\\
Krushnachandra Panigrahy \orcidC \href{https://orcid.org/0000-0003-0067-9298}{ \hspace{2mm}\textcolor{lightblue}{https://orcid.org/0000-0003-0067-9298}} \\
Jajati Keshari Sahoo \orcidC \href{https://orcid.org/0000-0001-6104-5171}{ \hspace{2mm}\textcolor{lightblue}{https://orcid.org/0000-0001-6104-5171}} \\
Yimin Wei \orcidD \href{https://orcid.org/0000-0001-6192-0546}{\hspace{2mm}\textcolor{lightblue}{https://orcid.org/0000-0001-6192-0546}}

\section*{Acknowledgments} 
The authors thank the handling editor and the referees for their very detailed comments, which helped to improve the manuscript.

\bibliographystyle{abbrv}
\bibliography{References}
\end{document}